\documentclass[10pt,a4paper,final]{amsart}

\usepackage{ifdraft}
\ifdraft{\usepackage[draft]{showkeys}}{\usepackage[final]{showkeys}}

\usepackage[T1]{fontenc} \usepackage[utf8x]{inputenc}

\usepackage{amsmath,amsfonts,mathscinet,eucal,amsthm,amssymb,bm,extarrows,upgreek,tensor,mathrsfs,textcomp,comment,mathtools,bbm,braket}

\usepackage{mathscinet}
\usepackage[shortlabels]{enumitem}

\newlist{enumeratesub}{enumerate}{3}
\setlist[enumeratesub]{label=(\alph*),ref=(Br\arabic{enumi}\alph*)}

\newlist{enumeratesubb}{enumerate}{3}
\setlist[enumeratesubb]{label=(\alph*),ref=(Br\arabic{enumi}\alph*)}

\usepackage{mathtools}
\mathtoolsset{mathic}

\numberwithin{equation}{section}

\usepackage[usenames,dvipsnames]{xcolor}

\usepackage{tikz}
\usetikzlibrary{patterns,matrix,through,arrows,decorations.pathreplacing}

\tikzset{anchorbase/.style={baseline={([yshift=-0.5ex]current bounding box.center)}},bullet/.style={circle,fill,inner sep=2pt}}

\usepackage{graphicx}

\usepackage{todonotes} 

\DeclareSymbolFont{usualmathcal}{OMS}{cmsy}{m}{n}
\DeclareSymbolFontAlphabet{\mathucal}{usualmathcal}

\theoremstyle{plain}
\newtheorem{theorem}{Theorem}[section]
\newtheorem*{theorem*}{Theorem}
\newtheorem{lemma}[theorem]{Lemma}
\newtheorem{proposition}[theorem]{Proposition}
\newtheorem{corollary}[theorem]{Corollary}
\newtheorem{conjecture}[theorem]{Conjecture}
\newtheorem{assumption}[theorem]{Assumption}

\theoremstyle{definition}
\newtheorem{definition}[theorem]{Definition}

\theoremstyle{example}

\theoremstyle{remark}
\newtheorem{remark}[theorem]{Remark}

\newcommand{\twomatrix}[4]{\ensuremath\begin{pmatrix} #1 & #2 \\ #3 & #4 \end{pmatrix}}

\newcommand{\N}{\mathbb{N}}
\newcommand{\Z}{\mathbb{Z}}

\newcommand{\C}{\mathbb{C}}

\newcommand{\gl}{\mathfrak{gl}}
\newcommand{\catO}{\mathcal{O}}

\newcommand{\suchthat}{\mid} 
\newcommand{\mapto}{\rightarrow}

\DeclareMathOperator{\Hom}{Hom}

\DeclareMathOperator{\End}{End}

\DeclareMathOperator{\image}{Im}

\newcommand{\id}{\mathrm{id}}

\newcommand{\abs}[1]{\left|#1\right|}

\renewcommand{\epsilon}{\varepsilon}

\renewcommand{\phi}{\varphi}

\newcommand{\up}{{\mathord\wedge}}
\newcommand{\down}{{\mathord\vee}}

\newcommand{\calC}{{\mathcal{C}}}
\newcommand{\calY}{{\mathcal{Y}}}
\newcommand{\sfs}{{\mathsf{s}}}
\newcommand{\sfw}{{\mathsf{w}}}

\newcommand{\frakh}{{\mathfrak{h}}}
\newcommand{\frakn}{{\mathfrak{n}}}
\newcommand{\frakp}{{\mathfrak{p}}}
\newcommand{\bbS}{{\mathbb{S}}}

\newcommand{\rmi}{{\mathrm{i}}}
\newcommand{\rmI}{{\mathrm{I}}}

\renewcommand{\down}{{\downarrow}}
\renewcommand{\up}{{\uparrow}}

\newcommand{\cycl}{{\mathrm{cycl}}}

\newcommand{\bolda}{{\boldsymbol{\mathrm{a}}}}
\newcommand{\boldi}{{\boldsymbol{\mathrm{i}}}}

\newcommand{\idem}{\boldsymbol{1}}
\newcommand{\idempotente}{{\boldsymbol{\mathrm{f}}}}
\newcommand{\bomega}{{\boldsymbol{\omega}}}
\newcommand{\VB}{\mathsf{V\hspace{-5.5pt}VBr}}

\newcommand{\Br}{{\mathrm{Br}}}
\newcommand{\uBr}{\Br}
\newcommand{\uVB}{\VB}
\newcommand{\udelta}{{\underline \delta}}

\newcommand{\hats}{{\hat s}}
\newcommand{\hate}{{\hat e}}
\newcommand{\ts}{{\sigma}}
\newcommand{\te}{{\tau}}
\newcommand{\thats}{{{\hat \sigma}}}
\newcommand{\thate}{{{\hat \tau}}}

\usepackage{cancel}

\newcommand{\Seq}{{\mathrm{Seq}}}

\usepackage[final=true]{hyperref,bookmark}

\hypersetup{ 
colorlinks = false, urlcolor = false,
pdfauthor =  {Antonio Sartori}, pdfkeywords = {Walled Brauer algebra, Degenerate affine Hecke algebra, Category O, Schur-Weyl duality}, 
  pdftitle = {The degenerate affine walled Brauer algebra},
 pdfsubject = {},
  pdfpagemode = UseNone, bookmarksopen = true, bookmarksopenlevel = 1, pdfdisplaydoctitle = true
}

\title[Walled Brauer algebras]{Walled Brauer algebras as idempotent truncations of level \(2\) cyclotomic quotients}
 \author{Antonio Sartori}
 \address{A.S.: Mathematisches Institut\\Universit\"at Freiburg\\Eckerstr.\ 1\\79104 Freiburg, Germany}
 \email{antonio.sartori@math.uni-freiburg.de}
 \author{Catharina Stroppel}
 \address{C.S.: Mathematisches Institut\\Universit\"at Bonn\\Endenicher Allee 60\\53115 Bonn, Germany}
 \email{stroppel@math.uni-bonn.de}
\thanks{A.S.\ was supported by the EPSRC grant EP/I014071. C.S.\ was supported by the Max-Planck Institute for Mathematics in Bonn when writing up this paper. She is very grateful for this exceptional possibibilty to work in this environment.}

\begin{document}




\begin{abstract}
We realize (via an explicit isomorphism) the walled Brauer algebra for an arbitrary integral parameter \(\delta\)  as an idempotent truncation of a level two cyclotomic degenerate affine walled Brauer algebra.
The latter arises naturally in Lie theory as the endomorphism ring of so-called mixed tensor products, i.e.\ of a parabolic Verma module tensored with some copies of the natural representation and its dual.  This provides us a method to construct central elements in the walled Brauer algebras and can be applied to establish the Koszulity of the walled Brauer algebra if \(\delta\not=0\). 
\end{abstract}



\maketitle
\tableofcontents

\section{Introduction}
\label{sec:introduction}

Let \(\delta \in \C\) be a fixed parameter.
The {\em walled Brauer algebra} \(B_{r,s}(\delta)\) is
a subalgebra of the classical
Brauer algebra \(B_{r+s}(\delta)\). This algebra was introduced independently by Turaev \cite{T} and Koike \cite{K}
in the late 1980s
motivated in part by a Schur-Weyl duality between
\(B_{r,s}(m)\) and the general linear group \(GL_m(\C)\)
arising from
mutually commuting actions on the ``mixed'' tensor space
\(V^{\otimes r} \otimes W^{\otimes s}\), where \(V\)
is the natural representation of \(GL_m(\C)\) and \(W := V^*\);
see also \cite{Beetal}. 

If \(\delta \notin \Z\) then the algebra \(B_{r,s}(\delta)\) is semisimple,
and its representation theory
can be described using character-theoretic methods
analogous to the ones used in the study of the complex representation theory
of the symmetric group; see e.g. \cite{Kosuda}, \cite{Ha}, \cite{N}. In case \(\delta\in\Z\) the algebra is in general not semisimple, but it was shown in \cite[Theorem 7.8]{MR2955190} that it again can be realized as the endomorphism ring of mixed tensor space, but now for the general linear Lie {\it superalgebra} \(\mathfrak{gl}(m|n)\)  for large enough integers \(m, n\), where \(\delta=m-n\) is the superdimension of the natural representation.    

As in the Okounkov-Vershik approach to the representation theory of the symmetric groups, \cite{CST}, the Jucys-Murphy elements play an important role and can be used to lift representations to the associated degenerate affine Hecke algebra. 
The \emph{degenerate affine walled Brauer algebra} was introduced independently in \cite{MR3244645} and \cite{RuiSu} playing the analogous role for the walled Brauer algebra as the degenerate affine Hecke algebra for the symmetric group. This algebra and its cyclotomic quotients were studied in \cite{MR3244645}, \cite{RuiSu}, \cite{Betal}. In this paper we consider a special case of a level \(2\) cyclotomic quotient
 \(\uVB_{r,t}(\boldsymbol \omega;\beta^\up_1,\beta^\up_2;\beta_1^\down,\beta_2^\down)\), see Definition \ref{def:4} and prove the main theorem:
  
\begin{theorem}
The walled Brauer algebra \(B_{r,s}(\delta)\) is isomorphic to an idempotent truncation \(\idempotente\uVB_{r,t}(\boldsymbol \omega;
  \beta^\up_1,\beta^\up_2;\beta_1^\down,\beta_2^\down)\idempotente\) of our chosen cyclotomic quotient. 
\end{theorem}

We want to stress that the theorem gives in fact an \emph{explicit} isomorphism, \eqref{eq:3}, between the idempotent truncation and the walled Brauer algebra. This is a nontrivial fact, since the idempotent truncation creates an interesting change in the parameters; the parameter \(\omega_0\) of the degenerate walled Brauer algebra does not coincide with the parameter \(\delta\) of the walled Brauer algebra.  In particular, the isomorphism does \emph{not} send standard generators to standard generators or to zero, but to a quite nontrivial expression involving inverses of square roots of formal power series. The paper contains therefore in Section~\ref{sec:invers-square-roots} a small general treatment about square root and inverses which we believe is of independent interest.  They allow us to make sense of the expressions defined in \eqref{eq:39} which then appear in the main isomorphism theorem. The main inspiration on the way of finding these formulas came from the impressive paper \cite{MR2235339} which allowed us to compare the Young orthogonal forms of the two algebras in the special cases when they are both semisimple.

As an application we construct, based on the arguments in \cite{MR3192543}, elements in the center of the walled Brauer algebra in terms of polynomials satisfying the Q-cancellation property. We conjecture that these elements generate the center. The result is slightly surprising, since the center is not generated (as for instance for the group algebra of the symmetric group) by the symmetric polynomials in the Jucys-Murphy elements as conjectured in \cite{MR2955190}, see Remark \ref{rem:2}. It would be interesting to realize this center as a cohomology ring of some variety in analogy to e.g. \cite{Bcenter}, \cite{StSpringer}, \cite{BLPW}. 

The main difficulty in the proof is to make sure that the proposed isomorphism is well-defined; see Section~\ref{sec:proofs}. This requires to verify the compatibility with the defining relations of the walled Brauer algebra which is done in several separate lemmas. Here one might prefer to  have a more economic presentation of the walled Brauer algebra (or walled Brauer category). Although a more elegant presentation can be found in \cite{Betal} we stick here to the more classical presentation which is far from being minimal. Despite the fact that we have to check a longer list of relations, each of them appears to be rather straight-forward, as soon as we have set up the correct framework in Section~\ref{sec:invers-square-roots}. Passing to fewer relations amounts to substantially more difficult proofs and less routine arguments for each of them. Moreover our arguments generalize directly to the Brauer algebra case, \cite{ESBrauer},  and we believe also to the quantized walled Brauer algebras from \cite{DDS} and to similar other examples of diagram algebras or other algebras of topological origin.  

The pure existence of an isomorphism as in the main theorem can be deduced from the results in \cite{MR2955190}, but requires a nontrivial passage between the walled Brauer algebras, the representation theory of the general linear super group and finally the generalized Khovanov arc algebras.  In this way it is impossible to make  the isomorphism explicit. 

By abstract nonsense our main theorem implies, see Remark  \ref{grading}, that the walled Brauer algebra can be equipped with a grading which can be realized in terms of the generalized Khovanov algebras from \cite{MR2955190}, but an explicit graded presentation expressed in the original standard generators is far from being obvious and will be dealt with in a separate paper.\\

{\bf Conventions.} In the following  we fix as ground field the complex numbers \(\C\). By an algebra we always mean an associative unitary algebra over \(\C\). Note that every finite-dimensional algebra comes with a unique maximal set of pairwise orthogonal primitive idempotents. There is then an easy correspondence between finite-dimensional algebras and \(\C\)--linear categories with a finite number of objects: if the algebra \(A\) has the set of pairwise orthogonal idempotents
\(\idem_1,\ldots,\idem_N\), then it is natural to identify \(A\) with a \(\C\)--linear category \(\calC\) with \(N\) objects, also denoted by \(\idem_1,\ldots,\idem_N\), and with homomorphism spaces
\begin{equation}
  \label{eq:76}
\Hom_{\calC} (\idem_j, \idem_\ell) =   \idem_j A \idem_\ell .
\end{equation}
In the following we will not distinguish between algebras and the corresponding categories. 
For the whole paper we fix natural numbers \(r,t \in \N\) and \(\delta \in \C\).\\

{\bf  Acknowledgment.} We like to thank Jonathan Brundan and Michael Ehrig for helpful discussions.

\section{The degenerate affine walled Brauer category}
\label{sec:degen-affine-wall}

We start by recalling the definition of the degenerate affine walled Brauer algebra. 

Denote by \(\bbS_n\) the symmetric group of permutations of \(n\) elements with its simple transpositions \(\sfs_k=(k,k+1)\) for \(k=1,\ldots,n-1\) as generators.
By an \((r,t)\)--sequence \(\bolda=(a_1,\ldots,a_{r+t})\) we mean a permutation of the sequence
    \begin{equation}
      \label{eq:1}
      (\underbrace{\up,\dotsc,\up}_{r},\underbrace{\down,\dotsc,\down}_{t}).
    \end{equation}
Let \(\Seq_{r,t}\) be the set of \((r,t)\)--sequences and \(J=\{1,\dots, r+t-1\}\).

\begin{definition}\label{def:7}
 The \emph{walled Brauer algebra} \(\uBr_{r,t}(\delta)\) is the \(\C\)--algebra on generators
    \begin{equation}
      \label{eq:339}
      \begin{aligned}
       & \idem_\bolda && \qquad \text{for all \(k\in J\) and \(\bolda\in\Seq_{r,t}\)},\\
      &  s_{k}\idem_\bolda& & \qquad \text{for all \(k\in J\) and \(\bolda\in\Seq_{r,t}\) such that } a_k=a_{k+1},\\ 
      &  \hat s_k\idem_\bolda, e_{k}\idem_\bolda, \hat e_k \idem_\bolda & & \qquad \text{for all \(k\in J\) and \(\bolda\in \Seq_{r,t}\) such that } a_k\neq a_{k+1},
      \end{aligned}
    \end{equation}
subject to the following relations (where we use \(\dot s_k,
    \dot e_k\) to denote both \(s_k, \hat s_k\) and \(e_k, \hat e_k\)
    respectively; the relations are assumed to hold for all possible
    choices that make sense with the convention that expressions like \(s_{k} \idem_\bolda\) are zero if $\bolda$ is not as in \eqref{eq:339}):
    \begin{enumerate}[label=(Br\arabic*)]
      \setlength{\itemsep}{1ex}
    \item \label{item:26} \(\{\idem_\bolda \suchthat \bolda \in \Seq_{r,t}\}\) is a complete set of pairwise orthogonal idempotents (in particular \(\sum_{\bolda \in \Seq_{r,t}} \idem_\bolda = 1\)),
    \item \label{item:27}
      \begin{enumeratesub}
      \item \label{item:3} \((s_k \idem_\bolda)\idem_\bolda = s_k \idem_\bolda\),  \((\hat s_k \idem_\bolda)\idem_\bolda = \hat s_k \idem_\bolda\),  \((e_k \idem_\bolda)\idem_\bolda = e_k \idem_\bolda\) and  \((\hat e_k \idem_\bolda)\idem_\bolda = \hat e_k \idem_\bolda\),
      \item \label{item:28} \(s_k \idem_\bolda = \idem_\bolda s_k \idem_\bolda\) and  \(e_k \idem_\bolda = \idem_\bolda e_k \idem_\bolda\),
      \item \label{item:29} \(\hat s_k \idem_\bolda = \idem_{\sfs_k \bolda} \hat s_k \idem_{\bolda}\) and \(\hat e_k \idem_\bolda = \idem_{\sfs_k \bolda} \hat e_k \idem_{\bolda}\),
      \end{enumeratesub}
  \item \label{item:br:1} \(s_k^2\idem_\bolda+\hat s^2_k \idem_\bolda= \idem_\bolda\),
    \item \label{item:br:2}
      \begin{enumeratesub}[]
      \item \label{item:br:3} \(\dot s_k \dot s_j=\dot s_j \dot s_k\) for
        \(\abs{k-j}>1\),
      \item \label{item:br:4} \(\dot s_k \dot s_{k+1} \dot s_k = \dot
        s_{k+1} \dot s_{k} \dot s_{k+1}\),
      \end{enumeratesub}
    \item \label{item:br:6}\(e_{k}^2= \delta e_{k} \),
      \item \label{item:br:8}
        \begin{enumeratesub}
        \item \label{item:br:9}\(\dot s_k \dot e_j=\dot e_j \dot s_k\) and
          \(\dot e_k \dot e_j = \dot e_j \dot e_k\) for \(\abs{k-j}>1\),
        \item \label{item:br:13}\(\hat{s}_k \dot e_k = \dot e_k = \dot e_k \hat{s}_k\),
        \item \label{item:br:14}\(\dot s_k \dot e_{k+1} \dot e_k = \dot
          s_{k+1} \dot e_k\) and \(\dot e_k \dot e_{k+1} \dot s_k =
          \dot e_k \dot s_{k+1}\), and
\(\dot e_{k+1} \dot e_k \dot s_{k+1} =
          \dot e_{k+1} \dot s_k\) and \(\dot s_{k+1} \dot e_k \dot
          e_{k+1} = \dot s_k \dot e_{k+1}\),
        \item \label{item:br:16}\(\dot e_{k+1} \dot e_k \dot e_{k+1} =
          \dot e_{k+1} \) and \(\dot e_k \dot e_{k+1} \dot e_k = \dot
          e_k\).
        \end{enumeratesub}
    \end{enumerate}
     Starting from \ref{item:br:1} we used here the abbreviation  \(s_k = \sum_{\bolda \in \Seq_{r,t}} s_k \idem_\bolda\), with the convention \(s_k \idem_\bolda = 0\) unless \(a_k = a_{k+1}\), and similarly for \(e_k, \hat s_k, \hat e_k\) with the vanishing convention unless \(a_k \not= a_{k+1}.\)
\end{definition}
Note that with our convention the product \(s_k\,\idem_\bolda\) coincides with the symbol \(s_k \idem_\bolda\). 

\begin{remark}\label{rem:7}
To make the relations more transparent, note that the walled Brauer algebra \(\Br_{r,t}(\delta)\) is the algebra on basis given by Brauer diagrams on \(2(r+t)\) vertices, with  additionally an orientation of each strand such that there are \(r\) upwards pointing strands and \(t\) downwards pointing strands; multiplication is given by vertical concatenation, where each closed circle is replaced by \(\delta\) (see Figures~\ref{fig:elements-of-br} and \ref{fig:relations-of-br}). In order to obtain the presentation given in Definition~\ref{def:7}, it is enough to consider generators and relations of the Brauer algebra (see for example \cite[Proposition~1.1]{MR1398116}) but equip them with orientations in  all possible ways, see \cite{ESosp} for the relation between the Brauer algebras and walled Brauer category. Note that the usual walled Brauer algebra from  \cite[Section 2]{MR2781018} is the subalgebra \(\idem_{\up^r \down^t} \Br_{r,t}(\delta) \idem_{\up^r \down^t}\).

It is possible to consider a more general walled Brauer category \(\mathucal{OB}\), which is called \emph{oriented Brauer category} in \cite{Betal}, by letting the number of strands vary and including oriented Brauer diagrams with a different number of source and target points (in informal words, one has additionally cups and caps, and our generators \(e_k\) and \(\hat e_k\) are obtained as composition of a cap and a cup). We remark that the category \(\mathucal{OB}\) contains all our algebras \(\Br_{r,t}(\delta)\) for \(r,t \geq 0\). Defining the category \(\mathucal{OB}\) as a monoidal category requires fewer generators and relations than our presentation (see \cite[Theorem~1.1]{Betal}). Nevertheless, our definition will be more handy for the purposes of our proofs.


\begin{figure}
  \centering
\begin{align*}
\idem_{\up\up\down\up} & = \,
  \begin{tikzpicture}[xscale=0.3,yscale=0.7,anchorbase]
    \draw[->] (1,0) -- ++(0,1);
    \draw[->] (2,0) -- ++(0,1);
    \draw[<-] (3,0) -- ++(0,1);
    \draw[->] (4,0) -- ++(0,1);
  \end{tikzpicture}&
 s_1 \idem_{\up\up\down\up} &= \,
  \begin{tikzpicture}[xscale=0.3,yscale=0.7,anchorbase]
    \draw[->] (1,0) -- ++(1,1);
    \draw[->] (2,0) -- ++(-1,1);
    \draw[<-] (3,0) -- ++(0,1);
    \draw[->] (4,0) -- ++(0,1);
  \end{tikzpicture}&
  s_2 \idem_{\down\up\up\up}& = \,
  \begin{tikzpicture}[xscale=0.3,yscale=0.7,anchorbase]
    \draw[<-] (1,0) -- ++(0,1);
    \draw[->] (2,0) -- ++(1,1);
    \draw[->] (3,0) -- ++(-1,1);
    \draw[->] (4,0) -- ++(0,1);
  \end{tikzpicture}& 
  \hat s_2 \idem_{\up\up\down\up}& = \,
  \begin{tikzpicture}[xscale=0.3,yscale=0.7,anchorbase]
    \draw[->] (1,0) -- ++(0,1);
    \draw[->] (2,0) -- ++(1,1);
    \draw[<-] (3,0) -- ++(-1,1);
    \draw[->] (4,0) -- ++(0,1);
  \end{tikzpicture}\\
 e_2 \idem_{\up\up\down\up} &= \,
  \begin{tikzpicture}[xscale=0.3,yscale=0.7,anchorbase]
    \draw[->] (1,0) -- ++(0,1);
    \draw[->] (2,0) arc (180:0:0.5 and 0.4);
    \draw[<-] (2,1) arc (180:360:0.5 and 0.4);
    \draw[->] (4,0) -- ++(0,1);
  \end{tikzpicture} & 
 \hat e_2 \idem_{\up\up\down\up}& = \,
  \begin{tikzpicture}[xscale=0.3,yscale=0.7,anchorbase]
    \draw[->] (1,0) -- ++(0,1);
    \draw[->] (2,0) arc (180:0:0.5 and 0.4);
    \draw[->] (2,1) arc (180:360:0.5 and 0.4);
    \draw[->] (4,0) -- ++(0,1);
  \end{tikzpicture}&
 e_3 \idem_{\up\up\down\up} &= \,
  \begin{tikzpicture}[xscale=0.3,yscale=0.7,anchorbase]
    \draw[->] (1,0) -- ++(0,1);
    \draw[->] (2,0) -- ++(0,1);
    \draw[<-] (3,0) arc (180:0:0.5 and 0.4);
    \draw[->] (3,1) arc (180:360:0.5 and 0.4);
  \end{tikzpicture}&
 \hat e_3 \idem_{\up\up\down\up}& = \,
  \begin{tikzpicture}[xscale=0.3,yscale=0.7,anchorbase]
    \draw[->] (1,0) -- ++(0,1);
    \draw[->] (2,0) -- ++(0,1);
    \draw[<-] (3,0) arc (180:0:0.5 and 0.4);
    \draw[<-] (3,1) arc (180:360:0.5 and 0.4);
  \end{tikzpicture}
\end{align*}
  \caption{Graphical version of some elements of $\Br_{3,1}(\delta)$.}
  \label{fig:elements-of-br}
\end{figure}
\begin{figure}
  \centering
\begin{gather*}
  \label{eq:50}
  \begin{tikzpicture}[xscale=0.3,yscale=0.5,anchorbase]
    \draw[->] (1,0) -- ++(1,1) -- ++(-1,1);
    \draw[->] (2,0) -- ++(-1,1) -- ++(1,1);
    \draw[<-] (3,0) -- ++(0,2);
    \draw[->] (4,0) -- ++(0,2);
  \end{tikzpicture} \, = \,
  \begin{tikzpicture}[xscale=0.3,yscale=0.5,anchorbase]
    \draw[->] (1,0) -- ++(0,2);
    \draw[->] (2,0) -- ++(0,2);
    \draw[<-] (3,0) -- ++(0,2);
    \draw[->] (4,0) -- ++(0,2);
  \end{tikzpicture}
\qquad
  \begin{tikzpicture}[xscale=0.3,yscale=0.5,anchorbase]
    \draw[->] (1,0) -- ++(0,2);
    \draw[->] (2,0) -- ++(1,1) -- ++(-1,1);
    \draw[<-] (3,0) -- ++(-1,1) -- ++(1,1);
    \draw[->] (4,0) -- ++(0,2);
  \end{tikzpicture} \, = \,
  \begin{tikzpicture}[xscale=0.3,yscale=0.5,anchorbase]
    \draw[->] (1,0) -- ++(0,2);
    \draw[->] (2,0) -- ++(0,2);
    \draw[<-] (3,0) -- ++(0,2);
    \draw[->] (4,0) -- ++(0,2);
  \end{tikzpicture}\,
\qquad
  \begin{tikzpicture}[xscale=0.3,yscale=0.5,anchorbase]
    \draw[->] (1,0) -- ++(0,2);
    \draw[->] (2,0) -- ++(1,1) arc (0:180:0.5 and 0.4) -- ++(1,-1);
    \draw[->] (3,2) arc (0:-180:0.5 and 0.4);
    \draw[->] (4,0) -- ++(0,2);
  \end{tikzpicture} \, = \,
  \begin{tikzpicture}[xscale=0.3,yscale=0.5,anchorbase]
    \draw[->] (1,0) -- ++(0,2);
    \draw[->] (2,0) arc (180:0:0.5 and 0.4);
    \draw[->] (3,2) arc (0:-180:0.5 and 0.4);
    \draw[->] (4,0) -- ++(0,2);
  \end{tikzpicture}\,\\[0.3cm]
  \begin{tikzpicture}[xscale=0.3,yscale=0.5,anchorbase]
    \draw[->] (1,0) -- ++(0,2);
    \draw[->] (2,0) arc (180:0:0.5 and 0.4);
    \draw[<-] (2,2) arc (180:360:0.5 and 0.4);
    \draw[<-] (3,1) arc (0:360:0.5 and 0.4);
    \draw[->] (4,0) -- ++(0,2);
  \end{tikzpicture} \, = \delta \,\,
  \begin{tikzpicture}[xscale=0.3,yscale=0.5,anchorbase]
    \draw[->] (1,0) -- ++(0,2);
    \draw[->] (2,0) -- ++(0,2);
    \draw[<-] (3,0) -- ++(0,2);
    \draw[->] (4,0) -- ++(0,2);
  \end{tikzpicture}\qquad\qquad
  \begin{tikzpicture}[xscale=0.3,yscale=0.5,anchorbase]
    \draw[->] (1,0) -- ++(0,2);
    \draw[->] (2,0) arc (180:0:0.5 and 0.4);
    \draw[<-] (2,2) arc (180:360:0.5 and 0.4);
    \draw[->] (3,1) arc (0:360:0.5 and 0.4);
    \draw[->] (4,0) -- ++(0,2);
  \end{tikzpicture} \, = \delta \,\,
  \begin{tikzpicture}[xscale=0.3,yscale=0.5,anchorbase]
    \draw[->] (1,0) -- ++(0,2);
    \draw[->] (2,0) -- ++(0,2);
    \draw[<-] (3,0) -- ++(0,2);
    \draw[->] (4,0) -- ++(0,2);
  \end{tikzpicture}\,\\[0.3cm]
    \begin{tikzpicture}[xscale=0.3,yscale=0.5,anchorbase]
    \draw[->] (1,0) -- ++(1,1) -- ++(1,1) -- ++(0,1);
    \draw[->] (2,0) -- ++(-1,1) -- ++(0,1) -- ++(1,1);
    \draw[<-] (3,0) -- ++(0,1) -- ++(-1,1) -- ++(-1,1);
    \draw[->] (4,0) -- ++(0,3);
  \end{tikzpicture} \, = \,
    \begin{tikzpicture}[xscale=0.3,yscale=0.5,anchorbase]
    \draw[->] (1,0) -- ++(0,1) -- ++(1,1) -- ++(1,1);
    \draw[->] (2,0) -- ++(1,1) -- ++(0,1) -- ++(-1,1);
    \draw[<-] (3,0) -- ++(-1,1) -- ++(-1,1) -- ++(0,1);
    \draw[->] (4,0) -- ++(0,3);
  \end{tikzpicture}\qquad
    \begin{tikzpicture}[xscale=0.3,yscale=0.5,anchorbase]
    \draw[->] (1,0) -- ++(0,3);
    \draw[->] (2,0) arc (180:0:0.5 and 0.4);
    \draw[<-] (2,3) arc (180:360:0.5 and 0.4);
    \draw[->] (4,0) -- ++(0,1) arc (0:180:0.5 and 0.4) arc (0:-180:0.5 and 0.4) -- ++(0,1) arc (180:0:0.5 and 0.4) arc (180:360:0.5 and 0.4) -- ++(0,1);
  \end{tikzpicture} \, = \,
  \begin{tikzpicture}[xscale=0.3,yscale=0.5,anchorbase]
    \draw[->] (1,0) -- ++(0,3);
    \draw[->] (2,0) arc (180:0:0.5 and 0.4);
    \draw[<-] (2,3) arc (180:360:0.5 and 0.4);
    \draw[->] (4,0) -- ++(0,3);
  \end{tikzpicture}\qquad
    \begin{tikzpicture}[xscale=0.3,yscale=0.5,anchorbase]
    \draw[->] (1,0) -- ++(1,1) arc (180:0:0.5 and 0.4) -- ++(0,-1);
    \draw[->] (2,0) -- ++(-1,1) -- ++(0,1) arc (180:0:0.5 and 0.4) arc (180:360:0.5 and 0.4) -- ++(0,1);
    \draw[->] (2,3) arc (0:-180:0.5 and 0.4);
    \draw[->] (4,0) -- ++(0,3);
  \end{tikzpicture} \, = \,
  \begin{tikzpicture}[xscale=0.3,yscale=0.5,anchorbase]
    \draw[->] (1,0) -- ++(0,1) arc (180:0:0.5 and 0.4) -- ++(1,-1);
    \draw[->] (2,0) -- ++(1,1) -- ++(0,2);
    \draw[<-] (1,3) arc (180:360:0.5 and 0.4);
    \draw[->] (4,0) -- ++(0,3);
  \end{tikzpicture}
\end{gather*}
    \caption{Graphical version of some relations of $\Br_{3,1}(\delta)$, namely (Br3), (Br5), (Br4b), (Br6d), (Br6c) respectively.}
  \label{fig:relations-of-br}
\end{figure}
\end{remark}

Because of the diagrammatics we call the sequences \(\bolda\) \emph{orientations}. Note that some orientations can make a relation trivial. For example, it is easy to see that \(e_k \hat e_{k+1} e_k = 0\).  The following says that that the so-called first Reidemeister relation on diagrams is equivalent to relation~\ref{item:br:16}:

\begin{lemma}
  \label{lem:50}
  Relation~\ref{item:br:16} can be replaced with
  \begin{equation}
    \label{eq:25}
     e_{k+1} s_k e_{k+1}\idem_\bolda = e_{k+1}\idem_\bolda \qquad \text{and} \qquad e_k s_{k+1} e_k \idem_{\bolda'} = e_k \idem_{\bolda'}
  \end{equation}
  for all \(\bolda,\bolda' \in\Seq_{r,t}\) with \(a_k=a_{k+1}\) and \(a'_{k+1}=a'_{k+2}\).
\end{lemma}

\begin{proof}
  Observe that by multiplying on both sides with \(\hat s_{k+1}\) we have
  \begin{equation}
    \label{eq:250}
    e_{k+1} s_k e_{k+1}\idem_\bolda = e_{k+1}\idem_\bolda  \iff \hat e_{k+1} s_k \hat e_{k+1} \idem_{\sfs_{k+1} \bolda} = e_{k+1} \idem_{\sfs_{k+1} \bolda}.
  \end{equation}

  First, let us check that \eqref{eq:25} holds. Indeed,
  \begin{equation*}
    \begin{aligned}
      e_{k+1} s_k e_{k+1} & \stackrel{\ref{item:br:16}}{=} e_{k+1} s_k \hat e_{k+1} e_k \hat e_{k+1} 
        & \stackrel{\ref{item:br:14}}{= }    e_{k+1} \hat s_{k+1} e_k \hat e_{k+1}  
        & \stackrel{\ref{item:br:13}}{= }    \hat e_{k+1}  e_k \hat e_{k+1}  
         \\& \stackrel{\ref{item:br:16}}{= }    e_{k+1}   
    \end{aligned}
  \end{equation*}
  Now suppose instead that \eqref{eq:25} holds. Then
  \begin{equation}
    \label{eq:249}
    \begin{aligned}
      e_{k+1} e_{k} e_{k+1}  & \stackrel{\ref{item:br:13}}{= } \hat e_{k+1} \hat s_{k+1} e_{k} e_{k+1} 
        & \stackrel{\ref{item:br:14}}{= }      \hat e_{k+1}   s_{k} \hat e_{k+1}   
         &\stackrel{\ref{eq:25}}{= }         e_{k+1},   
    \end{aligned}
  \end{equation}
  and we are done.
\end{proof}

The following allows us to  simplify slightly relation~\ref{item:br:4}:

\begin{lemma}
  \label{lem:52}
  Instead of \ref{item:br:4}, it suffices to impose \(s_k s_{k+1} s_k = s_{k+1} s_k s_{k+1}\) and either \(\hat s_k \hat s_{k+1} s_k =s_{k+1} \hat s_k \hat s_{k+1}\) or \(\hat s_k s_{k+1} \hat s_k = \hat s_{k+1} s_k \hat s_{k+1}\).
\end{lemma}

\begin{proof}
  It is immediate to check that the three possibilities given in the statement of the lemma are the only nontrivial orientations for the braid relation \ref{item:br:4}. Hence it is enough to see that the second one and the third one are equivalent. Indeed, we have
  \begin{equation}
    \label{eq:247}
    \begin{aligned}
      \hat s_k \hat s_{k+1} s_k =s_{k+1} \hat s_k \hat
      s_{k+1} & \iff \hat s_k \hat s_k \hat s_{k+1} s_k \hat
      s_{k+1} = \hat s_k s_{k+1} \hat s_k \hat s_{k+1} \hat
      s_{k+1} \\
      &\iff \hat s_{k+1} s_k \hat s_{k+1} = \hat s_k s_{k+1} \hat s_k.
    \end{aligned}
  \end{equation}
  as we wanted.
\end{proof}


We define the \emph{Jucys-Murphy elements} \(\{\xi_1\idem_\bolda,\dotsc,\xi_{r+t}\idem_\bolda\}\) of the walled Brauer category by setting \(\xi_1 \idem_\bolda=0\) and by the following recursive formulas:
\begin{equation}
  \label{eq:15}
  \xi_{k+1} \idem_\bolda =
  \begin{cases}
    s_k \xi_k s_k \idem_\bolda + s_k \idem_\bolda & \text{if } a_k=a_{k+1},\\
    \hat s_k \xi_k \hat s_k \idem_{\sfs_k \bolda} -  e_k \idem_{\sfs_k \bolda} &\text{if } a_k\neq a_{k+1}.
  \end{cases}
\end{equation}
where again  \(\xi_k = \sum_{\bolda \in \Seq_{r,t}} \xi_k \idem_\bolda\).

\begin{definition}\label{def:2}
  Let \(r,t \in \N\) and fix a sequence \(\boldsymbol
  \omega=(\omega_k)_{k \in \N}\) of complex parameters. The {\em
    degenerate affine walled Brauer algebra} \(\uVB_{r,t}(\boldsymbol
  \omega)\) is generated by the generators \eqref{eq:339} of the walled Brauer algebra and by elements \(y_i\) for \(1 \leq i \leq r+t\) subject to the above relations \ref{item:26}--\ref{item:br:8} and additionally to the following 
    \begin{enumerate}[resume,label=(Br\arabic*),ref=\arabic*]
      \setlength{\itemsep}{1ex}
      \item \label{item:8}
        \begin{enumeratesubb}
        \item \label{item:9} \(y_i \idem_\bolda = \idem_\bolda y_i\),
        \item \label{item:11}\(y_i y_j=y_j y_i\),
        \end{enumeratesubb}
    \item \label{item:2}
      \begin{enumeratesubb}
      \item \label{item:5}\(\dot s_k y_i= y_i \dot s_k\) for \(i \neq k,k+1\),
        \item \label{item:10}\(\dot e_k y_i= y_i \dot e_k\) for \(i \neq k,k+1\),
      \end{enumeratesubb}
      \item \label{item:7}\(e_{1} y_1^r e_{1} \idem_\bolda = \omega_r e_{1} \idem_\bolda\) for \(r \in \N\) if \((a_1,a_2 ) = {\uparrow \uparrow}\),
      \item \label{item:17}
        \begin{enumeratesubb}
        \item \label{item:18}\(s_k y_k\idem_\bolda - y_{k+1}s_k\idem_\bolda = - \idem_\bolda \idem_{\sfs_k \bolda}\) and \( s_k
          y_{k+1}\idem_\bolda -y_k s_k\idem_\bolda = \idem_\bolda \idem_{\sfs_k \bolda}\),
        \item \label{item:19}\(\hat s_k y_k - y_{k+1} \hat s_k = \hat e_k\) and
          \(\hat s_k y_{k+1} - y_k \hat s_k  = - \hat e_k\),
        \end{enumeratesubb}
      \item \label{item:20}
        \begin{enumeratesubb}
        \item \label{item:21}\(\dot e_k(y_k + y_{k+1})=0\),
        \item \label{item:22}\((y_k + y_{k+1}) \dot e_k = 0\).
        \end{enumeratesubb}
    \end{enumerate}
\end{definition}


\begin{remark}
  \label{rem:3}
  Note that by \cite[5.5]{Betal} our degenerate affine walled Brauer algebra is isomorphic to the corresponding algebra inside the affine oriented Brauer category \(\mathucal{AOB}(\omega_0,\omega_1,....).\) defined in \cite{Betal}.
\end{remark}

For \(\bolda \in \Seq_{r,t}\), we set
\begin{equation}
  \label{eq:46}
  \Br_\bolda(\delta) = \idem_\bolda \Br_{r,t}(\delta) \idem_\bolda \qquad \text{and} \qquad \uVB_\bolda(\bomega) = \idem_\bolda \uVB_{r,t}(\bomega) \idem_\bolda.
\end{equation}

\begin{remark}
  Let \(\omega_i = \delta\) for all \(i\geq 0\). Note that we have then a surjective homomorphism \(\VB_{r,t}(\bomega)
  \mapto \Br_{r,t}(\delta)\) which maps \(y_i\idem_\bolda \mapsto \xi_i \idem_\bolda\).\label{rem:4}
\end{remark}

Finally, we give the definition of cyclotomic quotients:

\begin{definition}
  \label{def:4}
  Given \(r,t \in \N\) and \(\boldsymbol \omega\) as above and additionally complex numbers \(\beta^\up_j,\beta^\down_j\) for \(j=1,\dotsc,l\)
then the {\em cyclotomic walled Brauer algebra} \(\uVB_{r,t}(\boldsymbol \omega;
  \beta^\up_j;\beta^\down_j)\) is the quotient of \(\uVB_{r,t}(\bomega)\) obtained by imposing on the
  degenerate affine walled Brauer category \(\uVB_{r,t}(\bomega)\) the following additional relations:
  \begin{align}
    (y_{1} - \beta^\up_1)(y_1 - \beta^\up_2)\dotsm(y_{1}-\beta^\up_l)\idem_{\bolda} = 0 & \quad \text{for every
    } \bolda \in \Seq_{r,t}
    \text{ with } a_1=\up, \label{eq:48}\\
    (y_{1} - \beta^\down_1)(y_1-\beta^\down_2)\dotsm(y_{1}-\beta^\down_l)\idem_\bolda = 0 & \quad \text{for
      every } \bolda \in \Seq_{r,t} \text{ with } a_1=\down. \label{eq:49}
  \end{align}
\end{definition}

The integer \(l\geq 0\) is called the \emph{level} of the cyclotomic quotient. We observe that in the cyclotomic quotient of level \(l\) the parameters \(\omega_j\) for \(j \geq l\) are uniquely determined:
\begin{lemma}
  \label{lem:5}
  Let \(r,t \in \N\), let \(l\) be a non-negative integer and let \(\omega_{j-1},\beta_j^\up,\beta_j^\down \in \C\) for \(j=1,\dotsc,l\). Then there is at most one choice of parameters \(\omega_h\) for \(h \geq l\) such that the cyclotomic walled Brauer algebra \(\uVB_{r,t}(\bomega;\beta_j^\up;\beta_j^\down)\) is non-zero.
\end{lemma}

\begin{proof}
  This is straightforward, since by relation~(Br\ref{item:7}) together with the cyclotomic equation~\eqref{eq:48} the parameters \(\omega_j\) must satisfy a linear recurrence relation of degree \(l\), whose characteristic polynomial is exactly \eqref{eq:48}.
\end{proof}

We do not claim however that the choice of the other parameters is free.

\section{Generalized eigenvalues and eigenspaces}
\label{sec:some-lemmas}

The elements \(\{y_k \idem_{\bolda} \suchthat \bolda \in \Seq_{r,t}, 1 \leq k \leq r+t\}\) generate a commutative subalgebra of \(\uVB_{r,t}(\bomega)\). We study their simultaneous (generalized) eigenspaces. 

Let \(M\) be a finite-dimensional (left) \(\uVB_{r,t}(\bomega)\)--module with its decomposition as \(M=\bigoplus_{\bolda \in \Seq_{r,t}} \idem_\bolda M\). Then \(M\) can be decomposed further into the direct sum of the generalized simultaneous eigenspaces
\begin{equation}
  \label{eq:4}
  M = \bigoplus_{\substack{\bolda \in \Seq_{r,t}\\\boldi \in \C^{r+t}}} \idem_\bolda M_\boldi
\end{equation}
where, for \(N \gg 0\) sufficiently large,  \((y_k  - \rmi_k)^N \idem_\bolda M_\boldi= 0\) for all \(k\).  

\begin{lemma}
  \label{lem:1}
  For all \(1 \leq k \leq r+t\) we have
  \begin{equation}
    \label{eq:5}
    s_k \idem_\bolda M_\boldi \subseteq \idem_\bolda M_\boldi + \idem_\bolda M_{\sfs_k \boldi}.
  \end{equation}
\end{lemma}

\begin{proof}
  Without loss of generality let \(k=1\) and \(r+t=2\).   Let \(B \subseteq \uVB_{r,t}(\bomega)\) be the subalgebra generated by \(s=s_1\idem_\bolda\), \(y_1 \idem_\bolda\), \(y_{2}\idem_\bolda\). Consider first the case of a vector \(v\) with proper eigenvalues (not just generalized eigenvalues) \(\rmi_1\) and \(\rmi_{2}\) for \(y_1\) and \(y_{2}\), respectively. Thanks to \ref{item:18}. the vectors \(v, s v\) span a \(B\)--submodule of \(M\). Let us assume that they are linearly independent (otherwise the claim is obvious). The action of the elements \(s\), \(y_1\) and \(y_{2}\) in the basis \(\{v, s v\}\) is then given by the matrices
  \begin{equation}
    \label{eq:6}
    s = \twomatrix{0}{1}{1}{0}, \qquad y_1 = \twomatrix{\rmi_1}{-1}{0}{\rmi_{2}}, \qquad y_{2} = \twomatrix{\rmi_{2}}{1}{0}{\rmi_1}.
  \end{equation}
  This implies that \(s v\) is a sum of two simultaneous eigenvectors for \(y_1,y_{2}\) with eigenvalues \(\rmi_1,\rmi_{2}\) and \(\rmi_{2},\rmi_1\), respectively.

Let us now turn to the general case.
 For \(N \geq 0\) and \(\boldi \in \C^2\) let \(\idem_\bolda M_\boldi^{(N)} = \{w \in \idem_\bolda M_\boldi \suchthat (y_1 - \rmi_1)^N w =0= (y_2- \rmi_2)^N w\}\). Now fix \(\boldi\) and let \(V^{(N)} = \idem_\bolda M_{\boldi}^{(N)} + \idem_\bolda M_{\sfs_1 \boldi}^{(N)}\). We show by induction that \(V^{(N)}\) is a \(B\)--submodule of \(\idem_\bolda M\), from which  our assertion follows. The case \(N=0\) is obvious and \(N=1\) is done already. So pick a vector \(v \in V^{(N+1)}\) but \(v \notin V^{(N)}\). Consider the quotient module \(\idem_\bolda M/V^{(N)}\). The image \(\overline v\) of \(v\) in this quotient is an eigenvector with eigenvalues \(\rmi_1,\rmi_2\). Hence we can apply the first paragraph of the proof to \(\overline v\), and the claim follows.
\end{proof}

\begin{lemma}
  \label{lem:3}
  For all \(\bolda \in \Seq_{r,t}\) and \(1 \leq k \leq r+t\) such that \(\sfs_k \bolda \neq \bolda\) we have
  \begin{equation}
    \label{eq:7}
    e_k \idem_\bolda M_\boldi \subseteq
    \begin{cases}
      \{0\} & \text{if } \rmi_k + \rmi_{k+1} \neq 0,\\
            \bigoplus_{\boldi'
        \in \rmI} \idem_\bolda M_{\boldi'} & \text{if } \rmi_k + \rmi_{k+1} = 0,
    \end{cases}
  \end{equation}
  where \(\rmI=\{\boldi' \in \C^{r+t} \suchthat \rmi'_j=\rmi_j \text{ for }j \neq k,k+1 \text{ and } \rmi_k + \rmi_{k+1} =0\} \). Analogously
  \begin{equation}
    \label{eq:8}
    \hat e_k \idem_\bolda M_\boldi \subseteq
    \begin{cases}
      \{0\} & \text{if } \rmi_k + \rmi_{k+1} \neq 0,\\
            \bigoplus_{\boldi'
        \in \rmI} \idem_{\sfs_k \bolda} M_{\boldi'} & \text{if } \rmi_k + \rmi_{k+1} = 0.
    \end{cases}
  \end{equation}
\end{lemma}

\begin{proof}
  This is an immediate consequence of relation \ref{item:21}.
\end{proof}

\begin{lemma}
  \label{lem:2}
  For all \(\bolda \in \Seq_{r,t}\) and \(1 \leq k \leq r+t\) such that \(\sfs_k \bolda \neq \bolda\) we have
  \begin{equation}
    \label{eq:9}
    \hat s_k \idem_\bolda M_\boldi \subseteq
    \begin{cases}
      \idem_{\sfs_k \bolda} M_{\sfs_k \boldi} & \text{if } \rmi_k + \rmi_{k+1} \neq 0,\\
            \bigoplus_{\boldi'
        \in \rmI} \idem_{\sfs_k \bolda} M_{\boldi'} & \text{if } \rmi_k+ \rmi_{k+1} = 0,
    \end{cases}
  \end{equation}
  where as before \(\rmI=\{\boldi' \in \C^{r+t} \suchthat \rmi'_j = \rmi_j \text{ for }j \neq k,k+1 \text{ and } \rmi'_k + \rmi'_{k+1}=0\} \).
\end{lemma}

\begin{proof}
  Again we may assume \(k=1\) and \(r=t=1\), and we remove subscripts from \(\hat s_1\), \(e_1\) and \(\hat e_1\). Note that in this case \(\rmI=\{(\rmi',-\rmi') \in \C^2\}\). Let \(v \in \idem_\bolda M_\boldi\). First, suppose \(\rmi_1 + \rmi_{2} \neq 0\). Then by relation \ref{item:19} and Lemma~\ref{lem:3} we have \(y_1\hat s v = \hat s y_2 v\) and \(y_2 \hat s v = y_1 \hat s v\), which implies that \(\hat s v \in \idem_{\sfs_1 \bolda} M_{\sfs_1 \boldi}\).

Now suppose that we are in the second case, that is \(\rmi_1+\rmi_2=0\). Let also \(E   \subseteq M\) be the submodule generated by the images of \(e\) and \(\hat e\). Note that by relations~\ref{item:br:13} and \ref{item:19} each element of \(E\) can be written as \(p(y_1) e w' + q(y_1) \hat e w''\) with \(p,q \in \C[y_1]\), and hence it follows by Lemma~\ref{lem:3} that \( E \subseteq \bigoplus_{\boldi' \in \rmI} \big(\idem_\bolda M_{\boldi'} + \idem_{\sfs_1 \bolda} M_{\boldi'}\big)\). Consider now the image \(\bar v\) of \(v\) in \(M/E\). Suppose \((y_1-\rmi_1)^N v= 0=(y_2+\rmi_1)^N v\). Then by \cite[eq.~(2.12)]{MR3244645} we have \((y_1+\rmi_1)^N \hat s \bar v = \hat s (y_2 +\rmi_1)^N \bar v=0\) and similarly \((y_2-\rmi_1)^N \hat s \bar v = \hat s (y_1 -\rmi_1)^N \bar v = 0\), whence \(\hat s \bar v\) has generalized eigenvalues \((-\rmi_1,\rmi_1)\). It follows that \(\hat s v \in \idem_{\sfs_1 \bolda} M_{(-\rmi_1,\rmi_1)} + E\), and we are done.
\end{proof}

We end this section with a rather technical result which will be important for the proof of our main theorem.

\begin{lemma}\label{lem:9}
  Let \(M\) be a finite-dimensional \(\uVB_{r,t}(\bomega)\)--module, and let \(\idem_\bolda M_\boldi\) be a generalized eigenspace of \(M\). For an index \(k\) suppose that either 
  \begin{enumerate}[(a)]
  \item \label{item:4} \(a_k=a_{k+1}\), \(\rmi_{k+1} \neq \rmi_k, \rmi_k \pm 1\)  and \((y_k - \rmi_{k+1}) \idem_\bolda M_{\sfs_k\boldi} = 0\), or
  \item \label{item:6} \(a_k \neq a_{k+1}\), \(\rmi_{k} + \rmi_{k+1} \neq 0\)  and \((y_k - \rmi_{k+1}) \idem_{\sfs_k \bolda} M_{\sfs_k\boldi} = 0\). 
  \end{enumerate}
  Then \((y_{k+1}-\rmi_{k+1}) \idem_\bolda M_\boldi = 0\), that is \(y_{k+1}\) has proper eigenvalue \(\rmi_{k+1}\) on \(\idem_\bolda M_\boldi\).
\end{lemma}

\begin{proof}
  By contradiction suppose that there is a nonzero vector \(v \in \idem_\bolda M_\boldi\) with \(w=(y_{k+1}-\rmi_{k+1})v \neq 0\). We can assume \((y_{k+1}- \rmi_{k+1})^2 v = 0\). First suppose we are in case~\ref{item:4}. Then by Lemma~\ref{lem:1} we know that \(s_{k} v \in  \idem_\bolda M_\boldi + \idem_\bolda M_{\sfs_k \boldi}\). 
By hypothesis we have \((y_k-\rmi_{k+1})(y_{k+1}-\rmi_{k+1})^N s_k v = 0\) for all \(N \gg 0\). But then we also have
  \begin{equation}
    \label{eq:58}
    \begin{aligned}
    0 & = (y_{k+1} - \rmi_{k+1})^N (y_k - \rmi_{k+1}) s_k v
   \\ &= (y_{k+1} - \rmi_{k+1})^N s_k (y_{k+1} - \rmi_{k+1})  v - (y_{k+1} - \rmi_{k+1})^N  v\\
    &  = s_k (y_k - \rmi_{k+1})^N w + \sum_{l+h=N-1} (y_{k+1}- \rmi_{k+1})^l (y_k - \rmi_{k+1})^h w \\
    &  = s_k (y_k - \rmi_{k+1})^N w +   (y_k - \rmi_{k+1})^{N-1} w ,
  \end{aligned}
\end{equation}
where for the third equality we used the relation from \cite[Lemma~2.9]{MR3244645}. Applying \(s_k\) we get \((y_k - \rmi_{k+1})^N w = -s_k (y_k - \rmi_{k+1})^{N-1} w\) for all \(N \gg 0\), and hence  \((y_k - \rmi_{k+1})^N w =  (y_k - \rmi_{k+1})^{N-2} w\). Recall that \(w \neq 0\), and note that since \(\rmi_{k+1}\) is different from \(\rmi_k\), which is the generalized eigenvalue of \(y_k\) on \(w\), the vectors \((y_k - \rmi_{k+1})^N w\) are always nonzero for all \(N\). In particular, for \(N \gg 0\), setting \(z = (y_k - \rmi_{k+1})^{N-2} w\) we get \((y_k - \rmi_{k+1})^2 z = z\). This implies that \(\pm 1\) is an eigenvalue of \((y_k - \rmi_{k+1})\), hence \(\rmi_{k+1} \pm 1\) is an eigenvalue of \(y_k\) on \(\idem_\bolda M_\boldi\). But since we are assuming \(\rmi_k \neq \rmi_{k+1} \pm 1\), this is a contradiction.

Now let us suppose we are in case~\ref{item:6}.
Then by Lemma~\ref{lem:2} above we have \(\hat s_k v\in \idem_{\sfs_k\bolda} M_{\sfs_k \boldi}\).
By hypothesis, \(y_k\) has proper eigenvalue \(\rmi_{k+1}\)  on \(\idem_{\sfs_k \bolda} M_{\sfs_k \boldi}\), hence \((y_k - \rmi_{k+1})\hat s_k v = 0\). But \((y_k - \rmi_{k+1})\hat s_k v = \hat s_k (y_{k+1}- \rmi_{k+1}) v + \hat e_k v = \hat s_k (y_{k+1}- \rmi_{k+1}) v\) by Lemma~\ref{lem:3}. Hence \(\hat s_k (y_{k+1}- \rmi_{k+1}) v = 0\) and also \((y_{k+1}- \rmi_{k+1})v = 0\), which gives a contradiction.
\end{proof}

\section{Inverses and square roots in finite-dimensional algebras}
\label{sec:invers-square-roots}

We summarize a few easy, but important facts which tell us when we are allowed to take inverses and square roots of elements of a finite dimensional algebra. We start by considering quotients of a polynomial ring.

\begin{lemma}
  \label{lem:4}
  Let \(f(x) \in \C[x]\), \(a \in \C\) and \(j\geq 0\). Then \(f(x)\) is invertible in \(\C[x]/(x-a)^j\)  if and only if \((x-a)\nmid f(x)\).
Moreover, in this case \(f(x)\) has  a square root in \(\C[x]/(x-a)^j\).
\end{lemma}

\begin{proof}
  The first claim is clear, since \(\C[x]/(x-a)^j\) is a local ring. Hence assume  \(f(x)\) is invertible in \(\C[x]/(x-a)^j\). By  translation it suffices to consider the case \(a = 0\). Since \(x \nmid f(x)\), we can suppose, up to a multiple, that \(f(x) = 1 + g(x)\) with \(x \mid g(x)\). Consider the formal power series 
  \begin{equation}
    \label{eq:10}
\sum_{n=0}^{\infty}\frac{(-1)^n(2n)!}{(1-2n)(n!)^2 4^n}t^n = 1 + \frac{1}{2}t - \frac{1}{8} t^2 + \frac{1}{16} t^3 - \frac{5}{128} t^4 + \cdots
  \end{equation}
 in \(\C[[t]]\). The expression~\eqref{eq:10} gives an explicit square root of \(1+t\). In \(\C[[t]]/t^j \cong \C[t]/t^j\) the sum becomes finite, and \eqref{eq:10} still gives a square root of \(1+t\). Substituting \(t=g(x)\) in \eqref{eq:10}, we  have a finite sum in \(\C[x]/x^j\) which squares to \(1+g(x)=f(x)\).
\end{proof}

The previous lemmas can be applied to arbitrary finite-dimensional algebra:

\begin{proposition}\label{prop:2}
Let still \(f(x) \in \C[x]\). Assume \(B\) is a finite-dimensional algebra, and let
  \(x_0 \in B\). Suppose that \((x-a) \nmid f(x)\) for all \(a \in \C\) which are generalized eigenvalues for the action of \(x_0\) on the regular representation. Then \(f(x_0)\in B\) has a (unique) inverse and a (non-unique) square root.
\end{proposition}

\begin{proof}
  Consider the commutative subalgebra \(C\subseteq B\) generated by \(x_0\). By elementary linear algebra, \(C\) is isomorphic to the quotient of a polynomial ring. Hence, if \(a_1,\dotsc,a_n\) are the generalized eigenvalues of \(x_0\), the Chinese Remainder Theorem implies
  \begin{equation}
    \label{eq:11}
    C \cong \C[t]/(t-a_1)^{j_1} \oplus \dotsb \oplus \C[t]/(t-a_n)^{j_n},
  \end{equation}
  as algebras via the assignment \(x_0 \mapsto t \oplus \dotsb \oplus t\). By assumption, Lemma~\ref{lem:4} can be applied to each summand. Hence the projection of \(f(t)\) onto each summand is invertible. In particular, \(f(x_0)\) is invertible in \(C\). Moreover, by Lemma~\ref{lem:4} each component has a square root,  and hence \(f(x_0)\) has a square root in \(C\). Observe that the inverse is unique (since if an element inside a non-necessarily-commutative ring has inverses on both sides then they agree and are unique), while for the square root we have at least \(n\) possibilities, corresponding to choosing a sign in each summand.
\end{proof}

Let as before \(x_0 \in B\) be an element of a finite-dimensional algebra, and let \(a \in \C\). If \(a\) is not a generalized eigenvalue of \(x_0\) then by Proposition~\ref{prop:2} we can write expressions like
\begin{equation}
\frac{1}{x_0-a}, \qquad \sqrt{x_0-a}, \qquad  \sqrt\frac{1}{x_0 -a}.\label{eq:52}
\end{equation}
As we remarked, the square root is not unique, but we  make one choice once and for all, so that for example \((\sqrt{x_0-a})^2=x_0-a\).

Suppose now that \(a\) is a generalized eigenvalue of \(x_0\). What we can do then is to consider the idempotent \(\eta_{x_0 \neq a}\) projecting onto the generalized eigenspaces of \(x_0\) with eigenvalues different from \(a\), and define
\begin{equation}
  \label{eq:53}
  \frac{1}{x_0-a}  = \frac{1}{\eta_{x_0 \neq a}(x_0 -a) \eta_{x_0 \neq a}}
\end{equation}
in the idempotent truncation \(\eta_{x_0 \neq a} B \eta_{x_0 \neq a}\), and similarly for the square root. Note that the notation \eqref{eq:53}, although handy, is extremely dangerous since for instance when we simplify we must remember  the idempotent:
\begin{equation}
  \label{eq:54}
  (x_0 -a) \,\frac{1}{x_0-a} = \eta_{x_0 \neq a}.
\end{equation}

\section{Cyclotomic quotients and category \texorpdfstring{$\catO$}{O}}
\label{sec:cycl-quot-categ}

Fix additionally two positive integers \(m,n \in \Z_{>0}\) and let now \(\delta \in \Z\).  Let \(\mathfrak{g}=\gl_{m+n}\) be the complex general linear Lie algebra with standard triangular decomposition \(\mathfrak{g} = \frakn^- \oplus \frakh \oplus \frakn^+\). Consider the standard parabolic subalgebra  \(\mathfrak p = (\gl_m \oplus \gl_n) + \mathfrak
n^+\) corresponding to the Levi subalgebra \(\gl_m
\oplus \gl_n \subseteq \gl_{m+n}\). Let \(\epsilon_1,\dotsc,\epsilon_{m+n}\) be the standard basis of \(\frakh^*\), and set
\begin{equation}
  \label{eq:35}
  \rho = -\epsilon_2 -2 \epsilon_3- \cdots - (m+n-1) \epsilon_{m+n}.
\end{equation}
Moreover define
\begin{equation}
  \underline \delta= -\delta (\epsilon_1 +
  \cdots + \epsilon_m).\label{eq:14}
\end{equation}


Let \(\catO(m,n)=\catO_{\textit{int}}^\frakp(\gl_{m+n})\) be the
integral parabolic 
BGG category \(\catO\) that is, see \cite{MR2428237}, the full subcategory of representations of $\mathfrak{g}$ consisting of finitely generated \(\gl_{n+m}\)--modules that are locally
finite over \(\mathfrak p\), semisimple over \(\mathfrak h\), and have all
integral weights.
This category is studied extensively in \cite{MR2781018}. It is easy to see that the elements from 
\begin{equation}
  \label{eq:18}
  \Lambda(m,n) = \left\{ \lambda \in \mathfrak h^* \, \left\vert\,
      \begin{aligned}
        &(\lambda + \rho, \epsilon_j) \in \Z \text{ for all } 1 \leq j\leq m+n,\\
        &(\lambda + \rho, \epsilon_1) > \cdots > (\lambda + \rho, \epsilon_m),\\
        &(\lambda + \rho, \epsilon_{m+1}) > \cdots > (\lambda + \rho,
        \epsilon_{m+n})
      \end{aligned}\right.
  \right\}
\end{equation}
are precisely the highest weights of the simple objects in  \(\catO(m,n)\). Given \(\lambda \in \Lambda(m,n)\), we denote by \(M^{\frakp}(\lambda)\) the parabolic Verma module with highest weight \(\lambda\). Let also \(V^\up=V\) and \(V^\down=V^*\) denote the standard representation of \(\mathfrak{g}\) and its dual, respectively. 
\begin{definition}
Let the \emph{cyclotomic parameters}  be
\begin{equation}\label{eq:56}
\begin{aligned}
      \beta_1^\up& = -\delta + \frac{m+n}{2}, &&& \beta^\up_2 &= \frac{n-m}{2},\\
          \beta^\down_1 & = \frac{m+n}{2}, &&&  \beta^\down_2 & =  \delta + \frac{m-n}{2},\\
          \omega_0 &= m+n, &&& \omega_1 &= - \delta m + \tfrac{(m+n)^2}{2}.
\end{aligned}
\end{equation}
\end{definition}
It follows from Lemma~\ref{lem:5} that there is at most one choice of the parameters \(\omega_j\) for \(j\geq 2\) such that the cyclotomic walled Brauer category \(\uVB_{r,t}(\bomega; \beta_1^\up,\beta_2^\up; \beta_1^\down, \beta_2^\down)\) is nontrivial. We fix from now on this choice and abbreviate
\begin{equation}
  \label{eq:28}
\VB_{r+t}^\cycl=\VB_{r+t}(\bomega;\beta_1^\up,\beta_2^\up;\beta_1^\down,\beta_2^\down) \qquad \text{and} \qquad \VB_\bolda^\cycl = \idem_\bolda \VB_{r+t}^\cycl \idem_\bolda.
\end{equation}

 We recall the main result of \cite{MR3244645}:
\begin{theorem}
  \label{thm:1}
  Suppose that \(m,n \geq r+t\). Then for  all $\bolda \in \Seq_{r,t}$ we have an isomorphism of algebras
  \begin{equation}
    \label{eq:19}
    \VB_\bolda^\cycl
    \cong \End_{\gl_{m+n}}(M^\frakp(\udelta) \otimes V^\bolda), 
  \end{equation}
  where \(V^\bolda=V^{a_1} \otimes \dotsb \otimes V^{a_{r+t}}\). In particular, \(\dim_\C \VB_{\bolda}^\cycl = 2^{r+t} (r+t)!\).
\end{theorem}


Note that \(M^{\frakp}(\udelta) \otimes V^\bolda\) has a Verma flag, i.e.\ a filtration with subquotients isomorphic to  parabolic Verma modules. Although the filtration is not unique, the multiplicity with which  a Verma module occurs is, see \cite[Theorem~9.8 (f)]{MR2428237}. Following \cite[Section~7]{MR3244645} we explain now the combinatorics of these multiplicities using 4-Young diagrams. (We explain it in detail, since there are a few inaccuracies in \emph{loc.\ cit.}). 


Recall that a \emph{Young diagram} is a collection of boxes arranged in left-justified rows with the number of boxes per row weakly decreasing from top to bottom. The \emph{content} of the box in the \(r\)--th row and \(c\)--th column (counting from the left to the right and from the top to the bottom, and starting with \(0\)) is \(r-c\).

A \emph{rotated Young diagram} is a Young diagram rotated by 180 degrees. The content of its boxes is by definition the content of the original box in the original Young diagram (Figure~\ref{part:fig:youngdiagrams}).

\begin{figure}
  \centering
  \begin{tikzpicture}[scale=0.7]
    \draw (0,0) -- (4,0);
    \draw (0,-1) -- (4,-1);
    \draw (0,-2) -- (3,-2);
    \draw (0,-3) -- (1,-3);
    \draw (0,0) -- (0,-3);
    \draw (1,0) -- (1,-3);
    \draw (2,0) -- (2,-2);
    \draw (3,0) -- (3,-2);
    \draw (4,0) -- (4,-1);
    \begin{scope}[xshift=0.5cm,yshift=-0.5cm]
      \node at (0,0) {\(0\)}; \node at (1,0) {\(-1\)}; \node at (2,0)
      {\(-2\)}; \node at (3,0) {\(-3\)}; \node at (0,-1) {\(1\)}; \node at
      (1,-1) {\(0\)}; \node at (2,-1) {\(-1\)}; \node at (0,-2) {\(2\)};
    \end{scope}
  \end{tikzpicture}
  \hspace{1cm}
  \begin{tikzpicture}[scale=0.7,rotate=180]
    \draw (0,0) -- (4,0);
    \draw (0,-1) -- (4,-1);
    \draw (0,-2) -- (3,-2);
    \draw (0,-3) -- (1,-3);
    \draw (0,0) -- (0,-3);
    \draw (1,0) -- (1,-3);
    \draw (2,0) -- (2,-2);
    \draw (3,0) -- (3,-2);
    \draw (4,0) -- (4,-1);
    \begin{scope}[xshift=0.5cm,yshift=-0.5cm]
      \node at (0,0) {\(0\)}; \node at (1,0) {\(-1\)}; \node at (2,0)
      {\(-2\)}; \node at (3,0) {\(-3\)}; \node at (0,-1) {\(1\)}; \node at
      (1,-1) {\(0\)}; \node at (2,-1) {\(-1\)}; \node at (0,-2) {\(2\)};
    \end{scope}
  \end{tikzpicture}
  \caption{A Young diagram and a rotated Young diagram, with the contents written in the boxes.}
  \label{part:fig:youngdiagrams}
\end{figure}
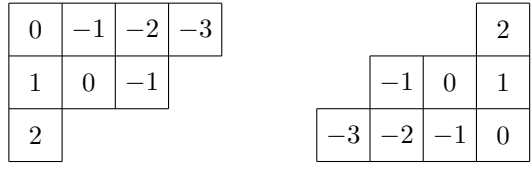

Consider now an infinite vertical strip consisting of \(m+n\) infinite columns, numbered \(m+n,m+n-1,\dotsc,2,1\) from the left to the right. Let \(v\) be the vertical line separating the leftmost \(n\) columns from the remaining \(m\) columns. Fix also a horizontal line \(o\). The lines \(o\) and \(v\) divide our strip into four regions. We define a {\em 4-Young diagram} to be a collection of boxes in this strip, such that in the two regions underneath the horizontal line \(o\) we have two Young diagrams and in the two regions above \(o\) we have two rotated Young diagrams, arranged such that no column contains boxes both above and below \(o\) (see Figure~\ref{part:fig:4youngdiagram}).

By definition a 4-Young diagram \(Y\) is the union of four Young diagrams. We define the content of a box in \(Y\) to be the content in its Young diagram, shifted by \(\beta_1^\down\), \(\beta_2^\up\), \(\beta_2^\down\) and \(\beta_1^\up\) as indicated in
Figure~\ref{part:fig:4youngdiagram}.

Given a 4-Young diagram \(Y\), let \(b_j(Y)\) be the number of boxes in the column \(i\) of \(Y\), multiplied by \(1\) or \(-1\) depending if the boxes are above or below the line \(o\). Then to the diagram \(Y\) we associate a weight \(w(Y) \in \Lambda(m,n)\) defined by 
\begin{equation}
  \label{eq:26}
  w(Y) = b_{m+n}(Y) \epsilon_{m+n} + b_{m+n-1}(Y) \epsilon_{m+n-1} + \dotsc b_1(Y) \epsilon_1.
\end{equation}


\begin{figure}
\centering
\begin{tikzpicture}[scale=0.7]

  \draw[dashed] (0,4)  -- (0,-4);
  \draw[dashed] (6,4)  -- (6,-4) node[right] {\(v\)};
  \draw[dashed] (12,4) -- (12,-4);
  \draw[dashed] (0,0) -- (12,0) node[above right] {\(o\)};

  \draw[dotted] (1,4) -- ++(0,-8);
  \draw[dotted] (2,4) -- ++(0,-8);
  \draw[dotted] (4,4) -- ++(0,-8);
  \draw[dotted] (5,4) -- ++(0,-8);
  \draw[dotted] (7,4) -- ++(0,-8);
  \draw[dotted] (8,4) -- ++(0,-8);
  \draw[dotted] (10,4) -- ++(0,-8);
  \draw[dotted] (11,4) -- ++(0,-8);

  \node[anchor=south east,rotate=310,font=\scriptsize] at (0.5,4) {\(m+n\)};
  \node[anchor=south east,rotate=310,font=\scriptsize] at (1.5,4) {\(m+n-1\)};
  \node[anchor=south east,rotate=310,font=\scriptsize] at (3,4) {\(\dotsb\)};
  \node[anchor=south east,rotate=310,font=\scriptsize] at (4.5,4) {\(m+2\)};
  \node[anchor=south east,rotate=310,font=\scriptsize] at (5.5,4) {\(m+1\)};
  \node[anchor=south east,rotate=310,font=\scriptsize] at (6.5,4) {\(m\)};
  \node[anchor=south east,rotate=310,font=\scriptsize] at (7.5,4) {\(m-1\)};
  \node[anchor=south east,rotate=310,font=\scriptsize] at (9,4) {\(\dotsb\)};
  \node[anchor=south east,rotate=310,font=\scriptsize] at (10.5,4) {\(2\)};
  \node[anchor=south east,rotate=310,font=\scriptsize] at (11.5,4) {\(1\)};

  \draw (0,0) -- (2,0);
  \draw (0,-1) -- (2,-1);
  \draw (0,-2) -- (1,-2);
  \draw (0,0) -- (0,-2);
  \draw (1,0) -- (1,-2);
  \draw (2,0) -- (2,-1);
  \node at (0.5,-0.5) {\(0\)};
  \node at (1.5,-0.5) {\(-1\)};
  \node at (0.5,-1.5) {\(1\)};
  \node[anchor=north west] at (2,-2) {\(\beta_1^\down = \frac{m+n}{2}\)};
  \draw[thin] (2,-2) -- (1.4,-1.4);

  \begin{scope}[xshift=6cm, rotate=180]
  \draw (0,0) -- (2,0);
  \draw (0,-1) -- (2,-1);
  \draw (0,-2) -- (1,-2);
  \draw (0,-3) -- (1,-3);
  \draw (0,0) -- (0,-3);
  \draw (1,0) -- (1,-3);
  \draw (2,0) -- (2,-1);
  \node at (0.5,-0.5) {\(0\)};
  \node at (1.5,-0.5) {\(-1\)};
  \node at (0.5,-1.5) {\(1\)};
  \node at (0.5,-2.5) {\(2\)};
  \node[anchor=south east] at (2,-2) {\(\beta_2^\up=\frac{n-m}{2}\)};
  \draw[thin] (2,-2) -- (1.4,-1.4);
  \end{scope}

  \begin{scope}[xshift=6cm]
  \draw (0,0) -- (2,0);
  \draw (0,-1) -- (2,-1);
  \draw (0,-2) -- (1,-2);
  \draw (0,-3) -- (1,-3);
  \draw (0,0) -- (0,-3);
  \draw (1,0) -- (1,-3);
  \draw (2,0) -- (2,-1);
  \node at (0.5,-0.5) {\(0\)};
  \node at (1.5,-0.5) {\(-1\)};
  \node at (0.5,-1.5) {\(1\)};
  \node at (0.5,-2.5) {\(2\)};
  \node[anchor=north west] at (2,-2) {\(\beta_2^\down=\frac{m-n}{2}+\delta\)};
  \draw[thin] (2,-2) -- (1.4,-1.4);
  \end{scope}

  \begin{scope}[xshift=12cm,rotate=180]
  \draw (0,0) -- (2,0);
  \draw (0,-1) -- (2,-1);
  \draw (0,-2) -- (1,-2);
  \draw (0,-3) -- (1,-3);
  \draw (0,0) -- (0,-4);
  \draw (1,0) -- (1,-4);
  \draw (2,0) -- (2,-1);
  \draw (0,-4) -- (1,-4);
  \node at (0.5,-0.5) {\(0\)};
  \node at (1.5,-0.5) {\(-1\)};
  \node at (0.5,-1.5) {\(1\)};
  \node at (0.5,-2.5) {\(2\)};
  \node at (0.5,-3.5) {\(3\)};
  \node[anchor=south east] at (2,-2) {\(\beta_1^\up = \frac{m+n}{2}-\delta\)};
  \draw[thin] (2,-2) -- (1.4,-1.4);
  \end{scope}

\end{tikzpicture}
\caption{A 4-Young diagram with the contents in the boxes. The corresponding weight is \(4 \epsilon_1 + \epsilon_2 - \epsilon_{m-1} - 3 \epsilon_m + 3\epsilon_{m+1}+\epsilon_{m+2} - \epsilon_{m+n-1} - 2 \epsilon_{m+n}\). }
\label{part:fig:4youngdiagram}
\end{figure}
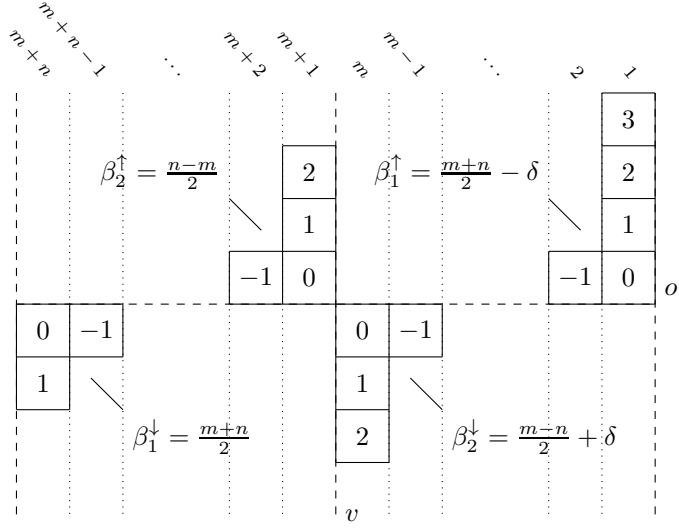
Given a 4-Young diagram \(Y\), we may obtain another 4-Young diagram
\(Y'\) by adding a box (in this case we also say that \(Y\) is obtained by
removing a box from \(Y'\)).  We will use the expressions \emph{adding} and \emph{removing
boxes} only if the result is again a 4-Young diagram.  For an
\((r,t)\)--sequence \(\bolda\) define \(\calY_\bolda\) to be the set 
\begin{equation}
 \{Y_\bullet =( Y_0,Y_1,\dotsc,Y_{r+t})\}\label{part:eq:236}
\end{equation}
of
sequences of 4-Young diagrams such that \(Y_0\) is the empty diagram and
\(Y_{i+1}\) is obtained from \(Y_j\) by
\begin{itemize}
\item adding a box above \(o\) or removing a box below \(o\) if \(a_j = 1\),
\item removing a box above \(o\) or adding a box below \(o\) if \(a_j=-1\).
\end{itemize}


\begin{proposition}[{\cite[Lemma~7.1 and Proposition~7.3]{MR3244645}}]
  \label{part:prop:4}
  Suppose \(m,n \geq r+t\). Then there is a bijection between \(\calY_\bolda\) and the parabolic Verma modules appearing in a Verma filtration of  \(M^\frakp(\udelta)\otimes V^{ \bolda}\)
  (counted with multiplicities) such that the following holds
  
  \begin{enumerate}[(a)]
  \item The Verma module corresponding to \(Y_\bullet=(Y_0,\dotsc,Y_{r+t})\) is isomorphic to
    \(M(\udelta + w(Y_{r+t}))\).
  \item For \(j = 1,\dotsc,r+t\) let \(\nu_j=1\) if
    \(Y_j\) is obtained from \(Y_{j-1}\) by adding a
    box of content \(i_j\), otherwise let \(\nu_j=-1\) if
    \(Y_j\) is obtained from \(Y_{j-1}\) by removing a
    box of content \(i_j\). Then the Verma module
    \(M(\udelta + w(Y_{r+t}))\) corresponding to \(Y_\bullet\)
    is contained in the generalized eigenspace for the
    \(y_k\)'s with generalized eigenvalues \((\nu_1
    i_1,\ldots, \nu_{r+t} i_{r+t})\).
  \end{enumerate}
\end{proposition}



\begin{definition}
\label{def:small}
We call a generalized eigenvalue for \(y_k\) \emph{small} if it corresponds to the content of  a box belonging to one of the two Young diagrams in the middle, adjacent to the vertical line \(v\). Otherwise we call it \emph{large}.
\end{definition}

For the rest of the paper, we make the following assumption:

\begin{assumption}\label{ass:1}
    We suppose that \(m\) and \(n\) are big enough and close
    enough to each other, in the sense that \(\abs{\beta_2^\up} + r +
    t < \beta_1^\up\) and \(\abs{\beta_2^\down} + r + t <
    \beta_1^\down\).
\end{assumption}

For example, one could choose \(m=n\), and \(r+t+\abs{\delta} < m\).
In the combinatorial partition calculus, this means that the two middle Young diagrams in the 4--Young diagram are always far away from the external Young diagrams, and the small eigenvalues are always smaller, in absolute value, than the large ones.

\begin{definition}
  For all \(k =1,\dotsc,r+t\) let \(\eta_k\idem_\bolda \in \VB_{\bolda}^\cycl\)
  be the idempotent projecting onto the generalized 
  eigenspaces of \(y_k \idem_\bolda\) with eigenvalues
  different from \(\beta^{a_k}_1\).
  Let also \(\idempotente_k\idem_\bolda = \eta_1 \eta_2 \dotsm
  \eta_k\idem_\bolda\) and \(\idempotente \idem_\bolda =
  \idempotente_{r+t} \idem_\bolda\). Finally, let
  \(\idempotente = \sum_{\bolda} \idempotente
  \idem_\bolda\).\label{def:1}
\end{definition}

Observe that \(\eta_k\idem_\bolda\) can be expressed as a
polynomial in \(y_k\idem_\bolda\); this will be important 
when we describe its relationship with elements in the algebra. 

\begin{remark}
  \label{rem:12}
  It follows by the partition calculus that if \(y_k \idem_\bolda\) has a large generalized eigenvalue on some composition factor, then there exists an index \(j \leq k\) such that \(y_j \idem_\bolda\) has generalized eigenvalue \(\beta_1^{a_j}\) on the same composition factor. In particular, \(\idempotente\) projects onto the generalized eigenspaces with small eigenvalues.
\end{remark}

Let \(\overline\catO(m,n) \subseteq \catO(m,n)\) be the full subcategory containing all simple modules \(L(\lambda)\) for \(\lambda \in \overline\Lambda(m,n)\), where
\begin{equation}
  \label{eq:20}
  \overline \Lambda(m,n) = \{ \lambda \in \Lambda(m,n) \suchthat -(m+n-1) \leq (\lambda+ \rho, \epsilon_j) \leq  0 \text{ for all }j\}.
\end{equation}
Note that this set is the union of orbits for the action of the Weyl group of $\mathfrak{g}$, hence it follows, \cite[4.9]{MR2428237} that \(\overline \catO(m,n)\) is a direct summand of \(\catO(m,n)\); in particular, there is a projection \(\pi: \catO(m,n) \mapto \overline\catO(m,n)\) and an inclusion \(\iota: \overline \catO(m,n) \mapto \catO(m,n)\).

By the partition calculus, an indecomposable summand \(S\) appearing in the decomposition of \(M(\udelta) \otimes V^\bolda\) as \(\mathfrak{g}\)--module has small generalized eigenvalues for the action of \(\uVB_{\bolda}^\cycl\) if and only if \(S \in \catO(m,n)\). Let \(F^{\updownarrow}: \overline \catO(m,n) \mapto \overline \catO(m,n)\) denote the functor \(\pi \circ (\otimes V^\updownarrow) \circ \iota\), where \(\updownarrow\) is either \(\downarrow\) or \(\uparrow\), and let \(F^\bolda=F^{a_{r+t}} \circ \dotsb \circ F^{a_{1}}\). As a direct consequence of Theorem ~\ref{thm:1} using the definitions we obtain:

\begin{corollary}
  \label{lem:7}
  The isomorphism \eqref{eq:19} induces an isomorphism
  \begin{equation}
    \label{eq:21}
    \idempotente\, \uVB_{\bolda}^\cycl\, \idempotente  \cong \End_{\mathfrak{g}}(F^\bolda M^\frakp(\udelta) ). 
  \end{equation}
\end{corollary}

\begin{lemma}
  \label{lem:8}
  For all \(\bolda \in \Seq_{r,t}\) we have
  \begin{equation}
    \label{eq:22}
\dim \End_{\mathfrak{g}}(F^\bolda M^\frakp(\udelta) )   = (r+t)!
  \end{equation}
\end{lemma}

\begin{proof}
  Note that we have pairs of biadjoint functors \((\otimes V^\up,\otimes V^\down)\) and \((\iota, \pi)\). Hence, the functors \(F^\up\) and \(F^\down\) are biadjoint as well. Let us now prove that the functors \(F^\up\) and \(F^\down\) commute. Consider \(X \in \overline\catO(m,n)\). Then of course \(\pi(X \otimes V \otimes V^*) \cong \pi(X \otimes V^* \otimes V)\). Now we have
  \begin{align}
    \label{eq:27}
    \pi(X \otimes V \otimes V^*) &= F^\down F^\up (X) \oplus \pi\big((1-\pi)(X \otimes V) \otimes V^*\big),\\
    \pi(X \otimes V^* \otimes V) &= F^\up F^\down (X) \oplus \pi\big((1-\pi)(X \otimes V^*) \otimes V\big).
  \end{align}
It follows by the theory of projective functors on category \(\catO\) (see in particular \cite[Theorem~7.8]{MR2428237}) that \(\pi\big((1-\pi)(X \otimes V) \otimes V^*\big) \cong X \cong \pi\big((1-\pi)(X \otimes V^*) \otimes V\big)\), and hence we must have \(F^\down F^\up(X) \cong F^\up F^\down(X)\).

Since \(F^\up\) and \(F^\down\) commute and are biadjoint, it is enough to consider the case \(a_k = {\uparrow}\) for all \(k\). Note that by our assumption the module \(M^\frakp(\udelta)\) is projective, since \(\udelta\) is maximal among all weights \(\lambda\) in its dot orbit such that the simple module \(L(\lambda)\) belongs to the parabolic category \(\overline \catO(m,n)\). By the same argument, all parabolic Verma which are composition factors of \((F^\up)^{r+t}(M^\frakp(\udelta))\) are also projective, hence \(\End_{\gl_{m+n}}((F^\up)^{r+t} M^\frakp(\udelta))\) is semisimple. Its dimension is then given by the following well-known counting formula: if the composition factors are \(M(\lambda_1),\dotsc,M(\lambda_N)\) and they appear \(\kappa_1,\dotsb,\kappa_N\) times respectively, then the dimension is \(\sum \kappa_j^2\).

In our case, by Proposition~\ref{part:prop:4} the summands appearing are parametrized by Young diagrams with \(r+t\) boxes (note that in this particular case we are only adding boxes in the middle Young diagram above the horizontal line \(o\)). The summand corresponding to the Young diagram \(Y\) occurs \(f_Y\) times, where \(f_Y\) is the number of paths of Young diagrams \(Y_0=\varnothing,Y_1,\dotsc,Y_{r+t}=Y\) such that \(Y_{k+1}\) is obtained by adding a box to \(Y_k\). It is well-known that \(f_Y\) is equal to the number of standard tableaux of shape \(Y\), and \(\sum_Y f_Y^2 = (r+t)!\), see \cite[Chapter~4]{MR1464693}, and this proves our claim.
\end{proof}

\section{The isomorphism theorem}
\label{sec:isomorphism-theorem}
We still assume that Assumption~\ref{ass:1} holds.
Recall that we defined \(\idempotente\idem_\bolda \in \VB_{\bolda}^\cycl\) to be the idempotent projecting onto the direct sum of the generalized eigenspaces with small eigenvalues, see Definitions~\ref{def:1} and ~\ref{def:small}.


We let \(\operatorname{rev}(\up)=\down\) and \(\operatorname{rev}(\down)=\up\).
Set
\begin{equation}
  b_k  = \sum_{\bolda} (\beta^{\operatorname{rev}(a_k)}_1 + y_k) \idem_\bolda, \qquad
  c_k  = \sum_\bolda (\beta^{a_k}_1 - y_k) \idem_\bolda \label{eq:105}
\end{equation}
and
\begin{equation}
Q_k  = \sqrt\frac{b_{k+1}}{b_k} \idempotente \label{eq:39}
\end{equation}
where for the definition of inverses and square roots we refer to  Section~\ref{sec:invers-square-roots}.


\begin{definition}
  \label{def:3}
  Let \(\bolda\) be an \((r,s)\)--sequence. Define
  \begin{align}
    \sigma_k \idem_\bolda & =
      - Q_k s_k  Q_k \idem_\bolda  + \frac{1}{b_k} \idempotente \idem_\bolda, &
    \hat \sigma_k \idem_\bolda & =
     - Q_k \hat s_k Q_k \idem_\bolda, \label{eq:13}\\
    \tau_k  \idem_\bolda& =
      Q_k e_k Q_k \idem_\bolda, &
    \hat \tau_k \idem_\bolda & =
     Q_k \hat e_k Q_k \idempotente \idem_\bolda.\label{eq:111}
  \end{align}
\end{definition}

We are now ready to state our main result, the following Isomorphism Theorem:

\begin{theorem}
  \label{thm:4}
  The assignments
  \begin{equation}
    \label{eq:3}
    \begin{aligned}
      \idem_\bolda & \longmapsto \idem_\bolda, &
      s_k\idem_\bolda & \longmapsto \sigma_k \idem_\bolda,
      & \hat s_k \idem_\bolda & \longmapsto \hat\sigma_k
      \idem_\bolda, \\
      & & e_k \idem_\bolda & \longmapsto \tau_k \idem_\bolda, & \hat e_k \idem_\bolda & \longmapsto
      \hat\tau_k \idem_\bolda
    \end{aligned}
  \end{equation}
define an isomorphism \(\Phi: \uBr_{r,s}(-\delta) \cong \idempotente\, \uVB_{r,s}^\cycl\,\idempotente\).
\end{theorem}
\begin{remark}
\label{theremark}
  The main feature of our isomorphism is the change of the
  walled Brauer parameter from \(\omega_0\) to \(\delta\),
  and indeed the most important relation that we will need
  to prove is \(\tau_k^2 \idem_\bolda = -\delta \tau_k
  \idem_\bolda\).  Essentially, this amounts to check that
  \begin{equation}
    \label{eq:30}
    e_k \frac{1}{\beta^{a_k}_1 + y_k} e_k \idem_\bolda= e_k \idem_\bolda,
  \end{equation}
  see Lemma~\ref{lem:20}. By
  \cite[Proposition~2.13]{MR3244645} (following an idea from
  \cite{MR1398116} further developed in \cite{MR2235339}) we
  can take a formal variable \(u\) and write
  \begin{equation}
    \label{eq:55}
    e_k \frac{1}{u-y_k} e_k \idem_\bolda = \frac{W^\bolda_k(u)}{u} e_k \idem_\bolda,
  \end{equation}
  where \(W^\bolda_k(u)\) is a formal power series in
  \(u^{-1}\). We may now be tempted to replace \(u =
  -\beta^{a_k}_1\) and be able to compute
  \(W^\bolda_k(-\beta^{a_k}_1)=\beta_1^{a_k}\), and hence
  obtain \eqref{eq:30} from \eqref{eq:55}. Now, while this
  can be made formal in the semisimple case (by using the
  eigenvalues of \(y_k\), as done several times in
  \cite{MR2235339}), it gets much more tricky in the
  non-semisimple case. Hence we need to take another way
  using the formalism from
  Section~\ref{sec:invers-square-roots}. We believe that our
  ideas can be useful for extending arguments from
  \cite{MR2235339} to the non-semisimple case.
  \label{rem:5}
\end{remark}

\begin{remark} \label{grading}By \cite[Theorem 6.10]{MR3244645} the algebra \(\uVB_{r,s}^\cycl\) can be equipped with the structure of a graded cellular and  Koszul algebra. Exactly the same arguments as in \cite[Section 5]{ES} imply then that  \(\uBr_{r,s}(-\delta) \cong \idempotente\, \uVB_{r,s}^\cycl\,\idempotente\) is graded cellular and, in case \(\delta\not=0\), even Koszul. 
\end{remark}

\begin{proof}[Proof of Isomorphism Theorem \ref{thm:4}]
  First, we show that the elements in Definition~\ref{def:3} satisfy the defining relations of the walled Brauer algebra, and so \eqref{eq:3} give a well-defined map \(\Phi:\uBr_{r,s}(-\delta) \cong \idempotente\,\uVB_{r,s}^\cycl\,\idempotente\).  The relations~\ref{item:26} and \ref{item:27} are straightforward. The relation~\ref{item:br:1} is given by Lemmas~\ref{lem:23} and \ref{lem:48}. The relation~\ref{item:br:3} is straightforward. For the relation~\ref{item:br:4}, by Lemma~\ref{lem:52} it is sufficient to consider two possible orientations, which we check in Lemmas~\ref{lem:24} and \ref{lem:35}. The relation~\ref{item:br:6} is given by Lemma~\ref{lem:20}. The relations~\ref{item:br:9} are obvious. The relation~\ref{item:br:13} is given by Lemma~\ref{lem:32}. The relations~\ref{item:br:14} are given by Lemmas~\ref{lem:33}, \ref{lem:34}, \ref{lem:37} and \ref{lem:38}. Finally, by Lemma~\ref{lem:50} the relation~\ref{item:br:16} can be replaced by \eqref{eq:25}, which we check in Lemmas~\ref{lem:21} and \ref{lem:25}.

We now show that \(\Phi\) is surjective. From \cite[Proposition~6.8]{MR3244645} we know that every element of \(\uVB_{r,s}^\cycl\) is a linear combination of elements of the form \(p_1 w p_2 \idem_\bolda\), where \(p_1,p_2 \in \C[y_1,\dotsc,y_{r+s}]\) with degree \( \leq 1\) in each variable and \(w = x_1 \dotsb x_N\) where \(x_\ell \in \{s_k, e_k, \hat s_k, \hat e_k \suchthat 1 \leq k \leq r+s\}\) for \(1 \leq \ell \leq N\). We call such a presentation \(x_1 \dotsc x_N \idem_\bolda\) for \(w \idem_\bolda\) a \emph{reduced word} if \(N\) is chosen minimally. 

Since by Lemma~\ref{lem:6} below all the elements \(y_k\idempotente \idem_\bolda\) are in the image of \(\Phi\), it suffices to show that \(\idempotente x_1 \dotsb x_N \idempotente \idem_\bolda \in \image \Phi\) for any reduced word \(x_1 \dotsb x_N \idem_\bolda\).
We show this by induction on the number \( N + \# \{k \suchthat x_k = \hat s_k\}\).


For the base case of the induction the claim is clear, since \(y_k \idempotente \idem_\bolda \in \image \Phi\), and so also all the polynomials expressions in the \(y_k\)'s (and in particular the elements \(b_k\idempotente \idem_\bolda\), their inverses and their square roots) are in the image of \(\Phi\), and then  it follows by inverting the expressions in \eqref{eq:13} and \eqref{eq:111} that also \(\idempotente s_k \idempotente\idem_\bolda\), \(\idempotente \hat s_k \idempotente \idem_\bolda\), \(\idempotente e_k \idempotente \idem_\bolda\) and \(\idempotente \hat e_k \idempotente \idem_\bolda\) are in the image of \(\Phi\).

Let us now turn to the inductive step. Consider a reduced word \(x_1 \dotsc x_N \idem_\bolda\) and assume \(x_N \in \{s_\ell, e_\ell, \hat s_\ell, \hat e_\ell\}\) for some \(1 \leq \ell \leq r+r\). By induction, we know that \(\idempotente x_1 \dotsc x_{N-1} \idempotente c_\ell \idempotente x_N \idempotente \idem_\bolda\) is in the image of \(\Phi\), since all three factors are. We can then move out the idempotent \(\idempotente\) thanks to Lemma~\ref{lem:49} below and we have 
\begin{equation}
  \label{eq:57}
\begin{aligned}[t]
&  \idempotente x_1 \dotsb x_{N-1} \idempotente c_\ell \idempotente x_N \idempotente \idem_\bolda =   \idempotente x_1 \dotsb x_{N-1}  c_\ell  x_N \idempotente  \idem_\bolda \\
 & \qquad=  \beta_1 \idempotente x_1 \dotsb x_{N-1}    x_N \idempotente  \idem_\bolda - \idempotente x_1 \dotsb x_{N-1}  y_\ell  x_N \idempotente  \idem_\bolda \\
 &   \qquad = \beta_1 \idempotente x_1 \dotsb x_{N-1}    x_N \idempotente  \idem_\bolda \pm
  \begin{cases}
    \idempotente x_1 \dotsb x_{N-1}    x_N y_j \idempotente  \idem_\bolda &\text{for some \(j\), or}\\
    y_j \idempotente x_1 \dotsb x_{N-1}    x_N  \idempotente  \idem_\bolda &\text{for some \(j\)}
  \end{cases}\\
  & \qquad \qquad + \text{smaller terms,}
\end{aligned}
\end{equation}
where \(\beta_1\) can be either \(\beta_1^\uparrow\) or \(\beta_1^\downarrow\), depending on \(\bolda\) and on \(x_N\).
Here the last equality is possible because the word was assumed to be reduced, hence the element \(y_\ell\) can be moved out to one of the two sides by using repeatedly relations (Br\ref{item:17}) and (Br\ref{item:20}). By looking at these relations, we see that the smaller terms do not contain any \(y_k\)'s, and are moreover either of length smaller than \(N\) or of length \(N\) but with strictly less \(\hat s_k\)'s than the word \(\idempotente x_1 \dotsb x_N \idempotente \idem_\bolda\). Hence by induction we can assume that all the smaller terms are contained in the image of \(\Phi\). Since \(\frac{1}{\beta \pm y_j}\idempotente \in \image \Phi\), it follows that \(\idempotente x_1 \dotsb x_N \idempotente \idem_\bolda \in \image \Phi\). This concludes the proof of surjectivity.

Finally, it follows that \(\Phi\) is an isomorphism by comparing dimensions: the dimension of \(\idem_\bolda \uBr_{r+t}(-\delta)\idem_{\bolda'}\) in the walled Brauer category is well known to be \((r+t)!\), \cite[(2.2)]{MR2781018}, and we are left to show that this is also the dimension of \(\idem_\bolda\idempotente\,\uVB_{r,t}^\cycl\,\idempotente \idem_{\bolda}\). If \(\bolda=\bolda'\)
then this follows from Corollary~\ref{lem:7} and Lemma \ref{lem:8}.
Otherwise, let  \(\sfw=\sfs_{i_1} \dotsm \sfs_{i_N}\) be a reduced expression of a permutation $\sfw$ such that \(\sfw\bolda = \bolda'\). Then pre-composing with \(\thats_{i_1} \dotsm \thats_{i_N}\idem_\bolda\) defines an isomorphism of vector spaces
  \begin{equation}
    \label{eq:24}
    \idem_\bolda \idempotente \,\uVB_{r,t}^\cycl\,\idempotente \idem_{\bolda'}  \cong      \idempotente \,\uVB_{\bolda}^\cycl\, \idempotente
  \end{equation}
  with inverse \(\thats_{i_N} \dotsm \thats_{i_1}\), and we are done.
\end{proof}

\begin{remark}
  \label{rem:6}
   Theorem~\ref{thm:4} can be easily extended to an isomorphism between the oriented Brauer category \(\mathucal{OB}\) (see Remark~\ref{rem:7}) and an idempotent truncation of the cyclotomic oriented Brauer category \(\mathucal{OB}^f\) of level 2 (which, as defined in \cite{Betal}, is a quotient of \(\mathucal{AOB}\), see Remark~\ref{rem:3}). Using the notation \(\mathfrak{c}_k\) and \(\mathfrak{d}_k\)  from \cite[Section~1]{Betal} for the cup and the cap with \(k\) strands to the left of the cup/cap (so that \(\hat e_k \idem_\bolda = \mathfrak{c}_k \mathfrak{d}_k \idem_\bolda\), where \(a_k,a_{k+1}=\down \up\)), one just needs to extend the isomorphism \eqref{eq:3} by setting
  \begin{equation}
    \label{eq:51}
    \idem_{\sfs_k \bolda} \mathfrak{c}_k \mapsto \idem_{\sfs_k \bolda} Q_k \mathfrak{c}_k  \qquad \text{and} \qquad \mathfrak{d}_k \idem_\bolda \mapsto \mathfrak{d}_k Q_k \idem_\bolda.
  \end{equation}
The only thing left to prove is that the adjunction relations \cite[(1.4) and (1.5)]{Betal} hold, but this is straightforward.
\end{remark}

\begin{lemma}
  \label{lem:6}
  We have \(\Phi\big((\xi_k-\beta_2^{a_k})\idem_\bolda\big) = -y_k\idempotente \idem_\bolda\) for all \(k\) and \(\bolda\).
\end{lemma}

\begin{proof}
  We prove the statement by induction. First, consider the case \(k=1\). Then \(- y_1 \idempotente \idem_\bolda = - \beta_2^{a_1} \idempotente \idem_\bolda\) and \(\Phi((\xi_1  -\beta_2^{a_k}) \idem_\bolda)= \Phi(- \beta_2^{a_1}\idem_\bolda  ) =-\beta_2^{a_1}\idempotente \idem_\bolda\). Let us now show the inductive step. First suppose that \(a_k=a_{k+1}\). Then
  \begin{equation}
    \label{eq:16}
    \begin{aligned}[t]
     & \Phi\big((\xi_{k+1}-\beta_2^{a_{k+1}})\idem_\bolda\big) =
      \Phi\big(s_k (\xi_k -\beta_2^{a_{k}}) s_k
      \idem_\bolda + s_k \idem_\bolda\big)\\
      &\qquad = - \sigma_k y_k \sigma_k \idem_\bolda + \sigma_k \idem_\bolda\\
      & \qquad = Q_k s_k Q_k y_k \sigma_k \idem_\bolda - \frac{1}{b_k}y_k \sigma_k \idem_\bolda + \sigma_k \idem_\bolda\\
      & \qquad = Q_k y_{k+1} s_k Q_k  \sigma_k \idem_\bolda - Q_k^2 \sigma_k \idem_\bolda - \frac{1}{b_k}y_k \sigma_k \idem_\bolda + \sigma_k \idem_\bolda\\
      & \qquad =   - y_{k+1}\left( - Q_k s_k Q_k     +\frac{1}{b_k}\right) \sigma_k \idem_\bolda + \frac{y_{k+1}}{b_k} \sigma_k \idem_\bolda - \frac{b_{k+1} + y_k-1}{b_k}  \sigma_k \idem_\bolda  \\
      & \qquad =- y_{k+1} \sigma_k^2 \idem_\bolda = -y_{k+1} \idempotente \idem_\bolda.
    \end{aligned}
  \end{equation}
  Otherwise, if \(a_k \neq a_{k+1}\) then
  \begin{equation}
    \label{eq:17}
    \begin{aligned}[t]
     & \Phi\big((\xi_{k+1}-\beta_2^{a_{k+1}})\idem_\bolda\big) =
      \Phi\big(\hat s_k (\xi_k  -\beta_2^{a_{k+1}}) \hat s_k
      \idem_\bolda - e_k \idem_\bolda\big)\\
      &\qquad = - \hat\sigma_k y_k \hat\sigma_k \idem_\bolda - \tilde e_k \idem_\bolda\\
      & \qquad = Q_k \hat s_k Q_k y_k \tilde {\hat s}_k \idem_\bolda - \tilde e_k \idem_\bolda\\
      & \qquad = Q_k y_{k+1} \hat s_k Q_k  \tilde {\hat s}_k \idem_\bolda + Q_k \tilde e_k Q_k  \tilde {\hat s}_k\idem_\bolda - \tilde e_k \idem_\bolda\\
      & \qquad =  - y_{k+1} \hat\sigma_k \tilde {\hat s}_k \idem_\bolda + \tilde{\hat e}_k  \hat\sigma_k \idem_\bolda- \tilde e_k \idem_\bolda\\
      & \qquad = -y_{k+1} \idempotente \idem_\bolda.
    \end{aligned}
  \end{equation}
  The lemma follows.
\end{proof}

\section{The center of the walled Brauer algebra}
\label{sec:center-walled-brauer}

As an application, we investigate the center of the walled Brauer category (and of the walled Brauer algebras). Let \(R=\C[y_1,\dotsc,y_{r+t}]\). We recall the following definition:
\begin{definition}[See also \cite{MR3192543}]
  We say that a polynomial \(p
  \in R\) satisfies the \emph{\(Q\)--cancellation property} with respect to the
  variables \(y_1,y_2\) if
  \begin{equation}
    \label{eq:o223}
    p(y_1,-y_1,y_3,\ldots,y_{r+t}) = p(0,0,y_3,\ldots,y_{r+t}).
  \end{equation}
  Analogously we say that \(p\) satisfies the \(Q\)--cancellation property with respect
  to the variables \(y_k,y_l\) if \(w\cdot p\) satisfies \eqref{eq:o223},
  where \(w \in \bbS_{r+t}\) is the permutation that exchanges \(1\) with
  \(k\) and \(2\) with \(l\) and \(\bbS_{r+t}\) acts on \(R\) permuting the
  variables.\label{def:o8}
\end{definition}

In \cite[Theorem~4.2]{MR3244645} it is shown that the  center of \(\uVB_{r,t}(\bomega)\) is isomorphic
 the subring of \((\bbS_r \times \bbS_t)\)--invariant polynomials \(p \in R^{{\bbS_r \times \bbS_{t}}}\) which satisfy the \(Q\)--cancellation property with respect to the variables \(y_r,y_{r+1}\). The isomorphism is given by the map
  \begin{equation}
    \label{eq:o258}
    p \mapsto \sum_{\bolda \in \Seq_{r,t}} (w_\bolda \cdot p) \id_\bolda,
  \end{equation}
where for each \(\bolda \in \Seq_{r,t}\) the element \(w_\bolda\) is a permutation such that \(w_\bolda \cdot (\up^r, \down^t) = \bolda\).

\begin{corollary}
  \label{cor:1}
  Let \(p \in R^{{\bbS_r \times \bbS_{t}}}\) be a polynomial which satisfies the \(Q\)--cancellation property with respect to the variables \(y_r,y_{r+1}\). Then the element
  \begin{equation}
    \label{eq:41}
    \sum_{\bolda \in \Seq_{r,t}} (w_\bolda \cdot p) (\xi_1,\dotsc,\xi_{r+t}) \id_\bolda
  \end{equation}
  is central in \(\Br_{r+t}(-\delta)\). In particular, \(p(\xi_1,\dotsc,\xi_{r+t})\) is central in the  walled Brauer algebra \(\Br_{\up^r,\down^t}(-\delta)\).
\end{corollary}

\begin{proof}
  The element~\eqref{eq:41} corresponds, under the isomorphism from Theorem~\ref{thm:4}, to the image of \(p\) under \eqref{eq:o258} in \(\uVB_{r,t}^\cycl\), which is central by \cite[Theorem~4.2]{MR3244645}.
\end{proof}

\begin{remark}\label{rem:2}
  In \cite[Remark~2.6]{MR2955190}, it is conjectured that
  any element in the center of the walled Brauer category
  can be expressed as a symmetric polynomial in the
  Jucys-Murphy elements. Although the definition of the
  Jucys-Murphy elements of \cite{MR2955190} is slightly
  different from ours, Corollary~\ref{cor:1} above suggests
  that the above mentioned conjecture points into the wrong
  direction. Indeed, we believe the following gives a
  counterexample.

  Consider the walled Brauer algebra \(\Br_{\up\up\down}(\delta)\). Note that in the notation of \cite{MR2955190} our \(\up\) corresponds to \(E\) and our \(\down\) corresponds to an \(F\). The Jucys-Murphy elements of \cite{MR2955190}, using the notation there, are 
  \begin{align}
    \label{eq:42}
    x_1^{EEF} & = 0, \\
    x_2^{EEF} & = \iota_{2,0}^{2,1} (x_2^{EE}) = \iota_{2,0}^{2,1} (x_{1,0}^{2,0}) =  \iota_{2,0}^{2,1} ( \overline{(1,2)}) = (1,2),\\
    x_3^{EEF} & = x_{2,0}^{2,1} = \overline{(1,3)} + \overline{(2,3)} = -(1,3) - (2,3).
  \end{align}
  In our notation, this translates as
  \begin{equation}
    \label{eq:43}
    x_2^{EEF} = s_1 \idem_{\up\up\down}, \qquad x_3^{EEF} = - s_1e_2s_1 \idem_{\up\up\down} - e_2 \idem_{\up\up\down}.
  \end{equation}
  It is easy to check that \(x_2^{EEF}+x_3^{EEF}=x_1^{EEF}+x_2^{EEF}+x_3^{EEF}\) is indeed central. On the other hand, if we consider the second elementary symmetric polynomial in the Jucys-Murphy elements we obtain 
  \begin{equation}
    \label{eq:44}
    x_2^{EEF}x_3^{EEF} = - e_2s_1\idem_{\up\up\down} - s_1 e_2\idem_{\up\up\down}
  \end{equation}
  and
  \begin{align}
    \label{eq:45}
    x_2^{EEF}x_3^{EEF} e_2 \idem_{\up\up\down} & = - e_2s_1e_2\idem_{\up\up\down} - s_1 e_2e_2\idem_{\up\up\down} = - e_2 \idem_{\up\up\down} - \delta s_1 e_2 \idem_{\up\up\down}, \\
    e_2 x_2^{EEF} x_3^{EEF} \idem_{\up\up\down} & = - e_2e_2s_1\idem_{\up\up\down} - e_2s_1 e_2\idem_{\up\up\down} = -\delta e_2 s_1 \idem_{\up\up\down} - e_2 \idem_{\up \up \down}.
  \end{align}
  Since \(e_2 s_1 \idem_{\up\up\down} \neq s_1 e_2 \idem_{\up\up\down}\), we have that \(x_2^{EEF}x_3^{EEF} \) does not commute with \(e_2 \idem_{\up\up\down}\). In our picture, this depends on the fact that the polynomial \(y_2y_3\) does not satisfy the \(Q\)--cancellation property with respect to \(y_2,y_3\).
\end{remark}

We conjecture the following:
\begin{conjecture}
  \label{conj:1}
  The center of the walled Brauer category \(\Br_{r,t}(\delta)\) is the subalgebra generated by the elements \eqref{eq:41}.
\end{conjecture}



\section{The proof of the main Theorem~\ref{thm:4}}
\label{sec:proofs}
We collect in this section the lemmas we used in the proof of our main theorem. If not stated explicitly we always assume the indices \(k-1,k,k+1\) appearing in the lemmas to be from $J$ such that the expressions make sense.

First, observe that by our partition calculus we have
\begin{equation}
  \label{eq:328}
  e_k \idempotente_{k+1} \idem_\bolda = e_k \idempotente_k \idem_\bolda.
\end{equation}
Indeed,  if the l.h.s.\ of \eqref{eq:328} on some composition factor is nontrivial then the generalized eigenvalue of \(y_{k+1}\) is the opposite of the generalized eigenvalue of \(y_k\). Forcing the generalized eigenvalue of \(y_k\) to be small also implies that the generalized eigenvalue of \(y_{k+1}\) is small.

Moreover, we note that  \(y_{k+1} \idempotente_k \idem_\bolda\) has only one generalized big eigenvalue, namely \(\beta_1^{a_k}\), and this  is a proper eigenvalue (that is, if \((y_{k+1} - \beta_1^{a_k})^N \idempotente_k \idem_\bolda v= 0\) for some \(v\) and some \(N \geq 0\) then \((y_{k+1} - \beta_1^{a_k}) \idempotente_k \idem_\bolda v= 0\)). Indeed, this is true for \(k=0\), and follows by induction using Lemma~\ref{lem:9}.

We stress that the idempotent \(\idempotente\) commutes with the element \(s_k\idem_\bolda\):

\begin{lemma}
  \label{lem:45}
  We have \(s_k \idempotente \idem_\bolda = \idempotente s_k \idem_\bolda\) for all \(\bolda\in\Seq_{r,t}\).
\end{lemma}

\begin{proof}
  This is clear, since it follows by Lemma~\ref{lem:1} that \(s_k\) sends generalized eigenspaces with small eigenvalues to generalized eigenspaces with small eigenvalues.
\end{proof}

Before going into details, we like to point out that our notation for the following proofs is a bit risky, although in our opinion it is the best we could figure out. The problem is that we write inverses and square roots of the \(b_k\)'s, but we are allowed to do that  only when \(b_k\) is next to the idempotent \(\idempotente\). Nevertheless, for the sake of readability, we will omit the idempotent as often as possible: for example, thanks to Lemma~\ref{lem:45}, we write the first equation of \eqref{eq:13} as
  \begin{equation}
    \label{eq:12}
        \sigma_k \idem_\bolda  =
      - \displaystyle \sqrt\frac{b_{k+1}}{b_k} s_k \sqrt\frac{b_{k+1}}{b_k}\idempotente\idem_\bolda  + \frac{1}{b_k} \idempotente \idem_\bolda.
  \end{equation}

\begin{lemma}
  \label{lem:19}
  Let $k\in J$ and let \(\bolda\) be an \((r,t)\)--sequence. Then the following formulas hold for all \(\bolda,\bolda' \in \Seq_{r,t}\) with  \(a_k = a_{k+1}\) and \(a'_k \neq a'_{k+1}\):
  \begin{subequations}
    \begin{align}
      s_k b_k\idem_\bolda &= b_{k+1} s_k\idem_\bolda -\idem_\bolda,
      &      s_k b_{k+1}\idem_\bolda &= b_k s_k \idem_\bolda + \idem_\bolda, \label{eq:171}\\
      s_k c_k \idem_\bolda&= c_{k+1} s_k\idem_\bolda +\idem_\bolda,
      & s_k c_{k+1} \idem_\bolda&= c_k s_k \idem_\bolda - \idem_\bolda,\label{eq:172}\\
      \hats_k b_k  \idem_{\bolda'} &= b_{k+1} \hats_k  \idem_{\bolda'} +\hate_k \idem_{\bolda'},
      & \hats_k b_{k+1}\idem_{\bolda'} &= b_k \hats_k \idem_{\bolda'} -\hate_k\idem_{\bolda'}, \label{eq:173}\\
      \hats_k c_k \idem_{\bolda'}&= c_{k+1} \hats_k\idem_{\bolda'} -\hate_k\idem_{\bolda'},
      &s_k c_{k+1}\idem_{\bolda'} &= c_k s_k \idem_{\bolda'} +\hate_k\idem_{\bolda'}, \label{eq:174}
    \end{align}
    \begin{align}
      s_k \frac{1}{b_k} \idempotente \idem_\bolda &= \frac{1}{b_{k+1}}s_k \idempotente \idem_\bolda+ \frac{1}{b_k b_{k+1}} \idempotente \idem_\bolda,\label{eq:29}\\
      s_k \frac{1}{b_{k+1}}  \idempotente \idem_\bolda&= \frac{1}{b_k} s_k \idempotente \idem_\bolda - \frac{1}{b_k b_{k+1}} \idempotente \idem_\bolda,\label{eq:31}\\
      s_k \frac{1}{c_k} \idempotente \idem_\bolda&= \frac{1}{c_{k+1}}s_k\idempotente \idem_\bolda- \frac{1}{c_k c_{k+1}}\idempotente\idem_\bolda,\label{eq:32}\\
      s_k \frac{1}{c_{k+1}}\idempotente\idem_\bolda &= \frac{1}{c_k} s_k \idempotente \idem_\bolda+ \frac{1}{c_k c_{k+1}}\idempotente \idem_\bolda,\label{eq:33}\\
      \idempotente \hats_k \frac{1}{b_k} \idempotente\idem_{\bolda'} &= \idempotente\frac{1}{b_{k+1}}\hats_k \idempotente \idem_{\bolda'}- \idempotente\frac{1}{b_{k+1}}\hate_k \frac{1}{b_{k}}\idempotente \idem_{\bolda'},\label{eq:34}\\
 \hats_k \frac{1}{b_{k+1}} \idem_{\bolda'}&= \frac{1}{b_k} \hats_k \idem_{\bolda'}+ \frac{1}{b_k}\hate_k \frac{1}{b_{k+1}} \idem_{\bolda'},\label{eq:36}\\
        \idempotente \hats_k \frac{1}{c_k}\idempotente \idem_{\bolda'} &= \idempotente\frac{1}{c_{k+1}}\hats_k\idempotente \idem_{\bolda'} +
      \idempotente\frac{1}{c_{k+1}}\hate_k \frac{1}{c_{k}}\idempotente \idem_{\bolda'},\label{eq:37}\\
 \idempotente\hats_k      \frac{1}{c_{k+1}}\idempotente \idem_{\bolda'}&= \idempotente\frac{1}{c_k} \hats_k \idempotente\idem_{\bolda'}-
     \idempotente \frac{1}{c_k}\hate_k \frac{1}{c_{k+1}}\idempotente \idem_{\bolda'}.\label{eq:38}
    \end{align}
  \end{subequations}
\end{lemma}

\begin{proof}
  Formulas \eqref{eq:171}, \eqref{eq:172}, \eqref{eq:173} and \eqref{eq:174} follow directly from the defining relations \ref{item:18} and \ref{item:19}.
The other formulas follow by multiplying by the idempotent \(\idempotente\) and by some easy algebraic manipulation. For example, \eqref{eq:29} follows from the first equation of \eqref{eq:171} by multiplying on the left by \(\frac{1}{b_{k+1}}\idempotente\) and on the right by \(\frac{1}{b_k}\idempotente\).
\end{proof}

\begin{lemma}\label{lem:49}
 The following equalities hold
  \begin{equation}
    \label{eq:290}
    c_k \idempotente e_k \idempotente = c_k e_k \idempotente, \qquad     c_k \idempotente \hat e_k \idempotente = c_k \hat e_k \idempotente, \qquad     c_k \idempotente \hat s_k \idempotente = c_k \hat s_k \idempotente.
  \end{equation}
\end{lemma}
\begin{proof}
  We have 
  \begin{equation}
c_k e_k
  \idempotente = c_k \idempotente_{k-1} e_k
  \idempotente = c_k \idempotente_k
  \idempotente_{k-1} e_k \idempotente + c_k (1
  -\idempotente_k) \idempotente_{k-1} e_k \idempotente =
  c_k \idempotente_k e_k \idempotente,\label{eq:40}
\end{equation}
since
  \(c_k (1 -\idempotente_k) \idempotente_{k-1}=0\)
  by definition and by our partition calculus. Here, we use the fact that \(\beta_1^{a_k}\) is the only big eigenvalue of \(y_k\idempotente \idem_\bolda\), and this has to be a proper eigenvalue, as we noticed in the discussion at the beginning of the section. Moreover, since \((y_k+y_{k+1})e_k=0\), the generalized eigenvalues of \(y_{k+1}\) on the image of \(e_k\) are determined by the generalized eigenvalues of \(y_k\). Since \(y_k \idempotente_k\) cannot have generalized eigenvalue \(-\beta_1^{a_{k+1}}\idem_\bolda\),
it follows that on the image of \(e_k \idem_\bolda\) the element \(y_{k+1}\idempotente_{k}\idem_\bolda \) cannot have generalized eigenvalue \(\beta_1^{a_{k+1}}\). In formulas, 
 we have \(\idempotente_k e_k \idem_\bolda= \idempotente_{k+1} e_k\idem_\bolda\).

This proves the first equality. The second one follows at once. The third one can be shown by the very same argument, using  Lemma~\ref{lem:2} for eigenvalue considerations.
\end{proof}

The following is the most crucial point of the proof:
\begin{lemma}
  \label{lem:20}
  We have \(\tilde e_k^2  = - \delta \tilde e_k \).
\end{lemma}

\begin{proof}
Let \(\bolda\) be such that \(a_k \neq a_{k+1}\). We compute using Lemma~\ref{lem:30} below:
  \begin{equation}
    \label{eq:110}
    \begin{aligned}
      \tilde e_k ^2 \idem_\bolda & =  Q_k e_k \frac{\beta_1^{a_{k}} + y_{k+1}}{b_k}\idempotente e_k Q_k \idem_\bolda =  Q_k e_k \frac{\beta_1^{a_{k}} - y_{k}}{b_k} e_k Q_k \idem_\bolda \\
&=  Q_k e_k \left(\frac{-\beta_1^{\operatorname{rev}(a_{k})} -y_k}{b_k} + \frac{\beta_1^{a_k}+\beta_1^{\operatorname{rev}(a_{k})}}{b_k}\right) e_k Q_k \idem_\bolda \\
& = -Q_k e_k^2 Q_k \idem_\bolda + (m+n-\delta)  Q_k e_k Q_k\idem_\bolda  = -\delta \tilde e_k \idem_\bolda.
    \end{aligned}
  \end{equation}
  Note that in the second equality we used Lemma~\ref{lem:49}.
\end{proof}

\begin{lemma}
  \label{lem:23}
  We have \(\sigma_k \sigma_k  \idem_\bolda= \idempotente \idem_\bolda \) for all \(\bolda\in\Seq_{r,t}\) with \(a_k=a_{k+1}\).
\end{lemma}

\begin{proof}
We compute
  \begin{equation}
    \label{eq:121}
    \begin{aligned}
      \sigma_k^2 \idem_\bolda & = \bigg(-Q_k s_k Q_k +
       \frac{1}{b_k} \idempotente \bigg)\bigg(-Q_k s_k Q_k + \frac{1}{b_k}\idempotente\bigg) \idem_\bolda\\
      &= Q_k s_k \frac{b_{k+1}}{b_k}\idempotente s_k Q_k\idem_\bolda -  \frac{1}{b_k}Q_ks_kQ_k \idem_\bolda- Q_ks_kQ_k \frac{1}{b_k}\idem_\bolda+\frac{1}{(b_k)^2}\idempotente\idem_\bolda.
    \end{aligned}
  \end{equation}
  Recall that \(\idempotente\) commutes with \(s_k\). We expand the first summand:
\begin{equation}
\label{eq:122}
\begin{aligned}
&Q_k \frac{1}{b_{k+1}} s_k b_{k+1}s_k Q_k \idem_\bolda + Q_k \frac{1}{b_k} s_k Q_k \idem_\bolda\\
&\qquad= Q_k \frac{b_k}{b_{k+1}} Q_k\idem_\bolda + Q_k \frac{1}{b_{k+1}} s_k Q_k\idem_\bolda + Q_k \frac{1}{b_k} s_k Q_k\idem_\bolda \\
&\qquad= \idempotente\idem_\bolda + Q_k s_k \frac{1}{b_k} Q_k\idem_\bolda - Q_k \frac{1}{b_kb_{k+1}} Q_k\idem_\bolda + Q_k \frac{1}{b_k} s_k Q_k\idem_\bolda.
    \end{aligned}
  \end{equation}
  Putting \eqref{eq:122} into \eqref{eq:121} we are done.
\end{proof}

\begin{lemma}
  \label{lem:48}
  We have \(\thats_k \thats_k \idem_\bolda = \idempotente \idem_\bolda\) for all \(\bolda\) with \(a_k \neq a_{k+1}\). 
\end{lemma}

\begin{proof}
 We compute
  \begin{equation}
    \label{eq:329}
    \begin{split}
      & \thats_k \thats_k \idem_\bolda  =  Q_k \hat s_k \idempotente \frac{b_{k+1}}{b_k}  \hat s_k Q_k  \idem_\bolda\\
    & \qquad  = Q_k \hat s_k \idempotente b_{k+1} \hat s_k \frac{1}{b_{k+1}} Q_k \idem_\bolda - Q_k \hat s_k \idempotente \frac{b_{k+1}}{b_k} \hat e_k \frac{1}{b_{k+1}} Q_k \idem_\bolda\\
    & \qquad  = Q_k \hat s_k \idempotente \hat s_k  \frac{b_k}{b_{k+1}}Q_k  \idem_\bolda - Q_k  \hat s_k \idempotente \hat e_k \frac{1}{b_{k+1}} Q_k \idem_\bolda
       - Q_k \hat s_k \idempotente \frac{b_{k+1}}{b_k} \hat e_k \frac{1}{b_{k+1}} Q_k \idem_\bolda.
    \end{split}
  \end{equation}
  Now note that
  \begin{equation}
    \label{eq:331}
    \hat s_k c_k \hat s_k \idem_\bolda = \hat s_k \hat s_k c_{k+1}\idem_\bolda - \hat s_k \hat e_k \idem_\bolda = c_{k+1} \idem_\bolda - e_k \idem_\bolda,
  \end{equation}
  and
  \begin{equation}
    \label{eq:332}
    \hat s_k \idempotente c_k \hat s_k \idem_\bolda= \hat s_k \idempotente \hat s_k c_{k+1}\idem_\bolda - \hat s_k \idempotente \hat e_k \idem_\bolda
  \end{equation}
  but since \(c_k \idempotente \hat s_k \idempotente\idem_\bolda = c_k \hat s_k \idempotente \idem_\bolda\), we have comparing \eqref{eq:331} and \eqref{eq:332}:
  \begin{equation}
    \label{eq:333}
    \idempotente \hat s_k \idempotente \hat s_k \idempotente \idem_\bolda = \idempotente \idem_\bolda - \idempotente e_k \frac{1}{c_{k+1}} \idempotente \idem_\bolda + \idempotente \hat s_k \idempotente \hat e_k \frac{1}{c_{k+1}} \idempotente \idem_\bolda.
  \end{equation}
  Now substituting \eqref{eq:333} in \eqref{eq:329} we get
  \begin{multline}
    \label{eq:334}
      \thats_k \thats_k \idem_\bolda  = 
 Q_k     \frac{b_k}{b_{k+1}} Q_k \idem_\bolda - Q_k
 e_k \frac{1}{c_{k+1}}   \frac{b_k}{b_{k+1}} Q_k \idem_\bolda +
Q_k  \hat s_k \idempotente \hat e_k  \frac{1}{c_{k+1}} \frac{b_k}{b_{k+1}} Q_k \idem_\bolda \\  - Q_k  \hat s_k \idempotente \hat e_k \frac{1}{b_{k+1}} Q_k \idem_\bolda - Q_k \hat s_k \idempotente \frac{b_{k+1}^*}{b_k} \hat e_k \frac{1}{b_{k+1}} Q_k \idem_\bolda.
  \end{multline}
  Note that the third and the fourth summand cancel together. Let us now compute
  \begin{multline}
    \label{eq:335}
     \idempotente \hat s_k \idempotente \frac{b_{k+1}}{b_k} \hat e_k
       \idem_\bolda\idempotente =  \idempotente \hat s_k  \frac{b_{k+1}}{b_k} \hat e_k  \idempotente     \idem_\bolda 
      =\idempotente b_k \hat s_k \frac{1}{b_k} \hat e_k  \idempotente\idem_\bolda - \idempotente \hat e_k \frac{1}{b_k} \hat e_k  \idempotente\idem_\bolda \\
      =\idempotente \frac{b_k}{b_{k+1}} \hat s_k \hat e_k  \idempotente\idem_\bolda - \idempotente \frac{b_k}{b_{k+1}} \hat e_k \frac{1}{b_k} \hat e_k \idempotente\idem_\bolda - \idempotente \hat e_k \frac{1}{b_k} \hat e_k  \idempotente\idem_\bolda        = - \idempotente e_k \idempotente \idem_\bolda.
  \end{multline}
  Substituting \eqref{eq:335} in \eqref{eq:334} we get then
  \begin{equation}
    \label{eq:336}
           \thats_k \thats_k \idem_\bolda  = 
Q_k      \frac{b_k}{b_{k+1}}Q_k \idem_\bolda -
Q_k  e_k \frac{1}{c_{k+1}}   \frac{b_k}{b_{k+1}} Q_k \idem_\bolda
 + Q_k  e_k \frac{1}{b_{k+1}} Q_k \idem_\bolda
   = \idempotente \idem_\bolda
  \end{equation}
  as we wanted.
\end{proof}

\begin{lemma}
  \label{lem:30}
  The following formula holds for all \(\bolda\in\Seq_{r,t}\) and $k\in J$:
  \begin{equation}
       e_k \frac{1}{b_k} e_k  \idempotente_{k-1} \idem_{\bolda} =   e_k \idempotente_{k-1} \idem_{\bolda}\label{eq:338}
     \end{equation}
\end{lemma}

\begin{proof}
  First, we observe that the formula makes sense. Indeed since  \(\idempotente_{k-1}\) commutes with \(e_k\), we have
  \begin{equation}
    \label{eq:277}
          e_k \frac{1}{b_k} e_k \idempotente_{k-1} \idem_\bolda =     e_k \frac{1}{b_k}\idempotente_{k-1} e_k \idempotente_{k-1} \idem_\bolda.
  \end{equation}
 Since \(\idempotente_{k-1}\) projects away from the generalized eigenspace of \(y_{k-1}\) with eigenvalue \(\beta^{\operatorname{rev}(a_k)}\), it follows from our partition calculus that it also projects away from the generalized eigenspace of \(y_k\) with eigenvalue \(-\beta^{\operatorname{rev}(a_k)}\). In particular, \(b_k \idempotente_{k-1}\) is invertible.

We prove the claim by induction on \(k\). First consider the case \(k=1\), and suppose \((a_1,a_{-1})={\uparrow \downarrow}\). Then it is easy to verify that
\begin{equation}
  \label{eq:324}
  \frac{1}{\beta^\down_1+y_1} \idempotente \idem_\bolda= -\frac{1}{n(m+n-\delta)}y_1 \idempotente\idem_\bolda + \frac{3n+m-2\delta}{2n(m+n-\delta)} \idempotente \idem_\bolda.
\end{equation}
Hence, recalling that \(e_1 y_1 e_1\idem_\bolda = \omega_1 e_1 \idem_\bolda \), we have
  \begin{multline}
  \label{eq:325}
    e_1 \frac{1}{\beta^\down_1+y_1} e_1 \idempotente \idem_\bolda= \left(- \frac{\omega_1}{n(m+n-\delta)} + \frac{\omega_0(3n+m-2\delta)}{2n(m+n-\delta)}\right) e_1 \idempotente \idem_\bolda\\
     = \frac{2 \delta m - (m+n)^2  + (m+n)(3n+m-2\delta)}{2n(m+n-\delta)} e_1 \idempotente \idem_\bolda = e_1 \idempotente \idem_\bolda.
  \end{multline}

Similarly if \((a_1,a_{-1})={\downarrow \uparrow}\) then
\begin{equation}
  \label{eq:326}
  \frac{1}{\beta^\up_1+y_1} = -\frac{1}{m(m+n-\delta)}y_1 + \frac{3m+n}{2m(m+n-\delta)}.
\end{equation}
Recalling that \(e_1 y_1 e_1 \idem_\bolda= \omega_1^* e_1\idem_\bolda \), where \(\omega_1^*=-\omega_1 + \omega_0^2\), we get
  \begin{multline}  \label{eq:327}
    e_1 \frac{1}{\beta^\up_1+y_1} e_1 \idempotente \idem_\bolda
= \left(- \frac{\omega_1-\omega_0}{m(m+n-\delta)} + \frac{\omega_0(3m+n)}{2m(m+n-\delta)}\right) e_1\idempotente \idem_\bolda\\
     = \frac{-2 \delta m -(m+n)^2  + (m+n)(3m+n)}{2m(m+n-\delta)} e_1 \idempotente \idem_\bolda= e_1\idempotente \idem_\bolda.
  \end{multline}

Let us now consider the inductive step.
Suppose first that \(a_k=a_{k+1}\).  Then
    \begin{multline}    \label{eq:280}
       e_{k+1} \frac{1}{b_{k+1}} e_{k+1} \idempotente_k \idem_\bolda  =  e_{k+1} s_k s_k
      \frac{1}{b_{k+1}}e_{k+1} \idempotente_k \idem_\bolda \\
      = e_{k+1} s_k \frac{1}{b_k} s_k e_{k+1}\idempotente_k\idem_\bolda - e_{k+1} s_k 
      \frac{1}{b_k b_{k+1}} e_{k+1} \idempotente_k \idem_\bolda.
    \end{multline}
  Since \(e_{k+1}=s_k \hat s_{k+1} e_k \hat s_{k+1} s_k\), and since
  \(y_k\) commutes with \(\hat s_{k+1}\), we can rewrite the first summand
  as
  \begin{multline}
    \label{eq:287}
e_{k+1} s_k \frac{1}{b_k} s_k e_{k+1}\idempotente_k\idem_\bolda   = s_k \hat s_{k+1} e_k \frac{1}{b_k} e_k \hat s_{k+1} s_k \idempotente_k \idem_\bolda \\=
     s_k \hat s_{k+1} e_k \hat s_{k+1} s_k \idempotente_k \idem_\bolda =  e_{k+1} \idempotente_k \idem_\bolda
  \end{multline}
  using the inductive hypothesis.
  For the second summand of \eqref{eq:280}, since
  \(y_{k+1}e_{k+1}=-y_{k+2}e_{k+1}\) and \(y_{k+2}\) commutes with \(s_k\),
  we can write
  \begin{equation}
    \label{eq:319}
    \begin{aligned}[t]
   - e_{k+1}s_k  \frac{1}{b_{k+1} b_{k}} e_{k+1} \idempotente_k \idem_\bolda & =- e_{k+1}     \frac{1}{b_{k+1}} s_k \frac{1}{b_k} e_{k+1} \idempotente_k \idem_\bolda\\
  & = - e_{k+1} s_k \frac{1}{b_k^2} e_{k+1} \idempotente_k \idem_\bolda +
    e_{k+1}\frac{1}{b_{k+1}b_k^2} e_{k+1} \idempotente_k \idem_\bolda, \\
  & = - e_{k+1} s_k  e_{k+1} \frac{1}{b_k^2} \idempotente_k \idem_\bolda +
    e_{k+1}\frac{1}{b_{k+1}b_k^2} e_{k+1} \idempotente_k \idem_\bolda, \\
 & = -   \frac{1}{b_k^2} e_{k+1}\idempotente_k \idem_\bolda + \frac{1}{b_k^2} e_{k+1}
    \frac{1}{b_{k+1}} e_{k+1} \idempotente_k \idem_\bolda.
    \end{aligned}
\end{equation}
  Putting all together, we obtain
  \begin{equation}
    \label{eq:322}
    \left( 1 - \frac{1}{b_k^2} \right) e_{k+1} \frac{1}{b_{k+1}}
    e_{k+1} \idempotente_k \idem_\bolda = \left(1 - \frac{1}{b_k^2} \right) e_{k+1} \idempotente_k \idem_\bolda.
  \end{equation}
Notice now that \(b_k\idempotente_k \idem_\bolda\) and \((1-b_k^2)\idempotente_k \idem_\bolda\) are invertible in the \(\idempotente_k \idem_\bolda\)--idempotent truncation, hence we can simplify on both sides and we are done.

Suppose now \(a_k\neq a_{k+1}\). Then we have
\begin{equation}
  \label{eq:323}
  \begin{aligned}[t]
   & e_{k+1} \frac{1}{b_{k+1}} e_{k+1}
    \idempotente_k \idem_\bolda = e_{k+1} \hat s_k \hat s_k
    \frac{1}{b_{k+1}} e_{k+1}
    \idempotente_k \idem_\bolda \\
    & \qquad = e_{k+1} \hat s_k 
    \frac{1}{b_k} \hat s_k e_{k+1}
    \idempotente_k \idem_\bolda + e_{k+1} \hat s_k \frac{1}{b_k} \hat e_k    \frac{1}{b_{k+1}} e_{k+1} \idempotente_k \idem_\bolda\\
    & \qquad
    \begin{multlined}
      = e_{k+1} \hat e_k \frac{1}{b_k} \hat e_k e_{k+1}
      \idempotente_k \idem_\bolda + e_{k+1}
      \frac{1}{b_{k+1}} \hat s_k \hat e_k e_{k+1}
      \frac{1}{c_k} \idempotente_k \idem_\bolda\\
      - e_{k+1}
      \frac{1}{b_{k+1}} \hat e_k \frac{1}{b_k} \hat e_k
      e_{k+1} \frac{1}{c_k} \idempotente_k \idem_\bolda 
\end{multlined}\\
    & \qquad =
      e_{k+1}  e_k e_{k+1} \idempotente_k \idem_\bolda + e_{k+1}
       \hat s_k \hat e_k e_{k+1}
      \frac{1}{c_k^2} \idempotente_k \idem_\bolda  -      e_{k+1}   e_k       e_{k+1}
      \frac{1}{c_k^2} \idempotente_k \idem_\bolda\\
      & \qquad = \left(1 +\frac{1}{c_k^2} - \frac{1}{c_k^2} \right)e_{k+1} \idempotente_k \idem_\bolda.
  \end{aligned}
\end{equation}
The claim follows.
\end{proof}

\begin{lemma}
  \label{lem:47}
  We have for the formula
  \begin{equation}
    \label{eq:310}
        \sigma_k \idem_\bolda  =
            -\idempotente\sqrt\frac{b_{k}}{b_{k+1}} s_k \sqrt\frac{b_{k}}{b_{k+1}}\idempotente \idem_\bolda - \frac{1}{b_{k+1}}\idempotente \idem_\bolda .
  \end{equation}
\end{lemma}

\begin{proof}
Multiplying the first of \eqref{eq:29} by \(\sqrt{b_{k+1}}\) from the left and by \(\sqrt{b_{k}}\) from the right we get
  \begin{equation}
    \label{eq:316}
    \idempotente \sqrt{b_{k+1}} s_k \frac{1}{\sqrt{b_{k}}}\idempotente \idem_\bolda = \idempotente \frac{1}{\sqrt{b_{k+1}}} s_k \sqrt{b_k}\idempotente \idem_\bolda + \frac{1}{\sqrt{b_k b_{k+1}}} \idempotente \idem_\bolda.
  \end{equation}
  Similarly from the second of \eqref{eq:31} we get
  \begin{equation}
    \label{eq:317}
    \idempotente \sqrt{b_{k}} s_k \frac{1}{\sqrt{b_{k+1}}}\idempotente \idem_\bolda = \idempotente\frac{1}{\sqrt{b_{k}}} s_k \sqrt{b_{k+1}}\idempotente \idem_\bolda - \frac{1}{\sqrt{b_k b_{k+1}}} \idempotente \idem_\bolda.
  \end{equation}
  We can then compute using \eqref{eq:316} and \eqref{eq:317}:
  \begin{equation}
    \label{eq:318}
    \begin{aligned}[t]
   &      - \sqrt\frac{b_{k+1}}{b_k} s_k
      \sqrt\frac{b_{k+1}}{b_k} \idempotente\idem_\bolda +
      \frac{1}{b_k}\idempotente \idem_\bolda \\
      &\qquad \qquad=
      - \sqrt\frac{1}{b_k b_{k+1}} s_k  \sqrt{b_{k+1} b_k}\idempotente\idem_\bolda - \frac{1}{b_k } \idempotente\idem_\bolda + \frac{1}{b_k}\idempotente\idem_\bolda \\
     & \qquad\qquad= - \sqrt\frac{b_k}{ b_{k+1}} s_k
      \sqrt{\frac{b_k}{b_{k+1} }} \idempotente\idem_\bolda -
      \frac{1}{b_{k+1}} \idempotente \idem_\bolda.
    \end{aligned}
  \end{equation}
  The claim follows.
\end{proof}

\begin{lemma}
  \label{lem:31}
  The following holds for \(k, k+1\in J\) and \(\bolda \in\Seq_{r,t}\):
  \begin{equation}
    e_k s_{k+1} \frac{1}{b_k} e_k\idempotente\idem_\bolda  = e_k \frac{1}{b_k} s_{k+1} e_k \idempotente\idem_\bolda = 0.\label{eq:179}
\end{equation}
\end{lemma}

\begin{proof}
  Assume \(a_{k+1}=a_{k+2}\), otherwise the asserted identity is trivial. We have
  \begin{multline}
    \label{eq:114}
      e_k s_{k+1} \frac{1}{b_k} e_k \idempotente\idem_\bolda  = e_k s_{k+1}
      \frac{1}{c_{k+1}} e_k \idempotente \idem_\bolda\\
  = e_k \frac{1}{c_{k+2}} s_{k+1} e_k\idempotente \idem_\bolda -
      e_k \frac{1}{c_{k+1}c_{k+2}} e_k \idempotente\idem_\bolda 
       = \frac{1}{c_{k+2}}\left( e_k - e_k \frac{1}{b_{k}} e_k \right)\idempotente \idem_\bolda= 0,
  \end{multline}
  where we use relation~\ref{item:22} of Definition~\ref{def:2} in the first and in the third equalities, \eqref{eq:32} in the second equality and \eqref{eq:338} in the last equality.
\end{proof}

\begin{lemma}
  \label{lem:32}
  We have \(\thats_k \te_k \idem_\bolda = \thate_k \idem_\bolda\) and \(\te_k \thats_k \idem_\bolda= \thate_k \idem_\bolda \) for any \(\bolda\).
\end{lemma}

\begin{proof}
  We check only the first equality. The other one can be showed in the same way. We have:
  \begin{equation}
    \label{eq:184}
    \begin{aligned}[t]
       \thats_k \te_k  &=  - Q_k \hats_k \idempotente \frac{b_{k+1}}{b_k} e_k
      Q_k    = - Q_k \hats_k  \frac{b_{k+1}}{b_k} e_k Q_k   \\
       &= - Q_k (b_k
      \hats_k -\hate_k) \frac{1}{b_k} e_k Q_k  \\
      & = - Q_k \left(\frac{b_k}{b_{k+1}} 
      \hats_k e_k - \frac{b_k}{b_{k+1}} \hate_k \frac{1}{b_k} e_k -\hate_k \frac{1}{b_k}e_k \right)  Q_k  \\
      & = - Q_k \left(\frac{b_k}{b_{k+1}} 
      \hate_k - \frac{b_k}{b_{k+1}} \hate_k -\hate_k \right)  Q_k  = \thate_k  .\\
    \end{aligned}
  \end{equation}
  In the second equality, we got rid of the idempotent using Lemma~\ref{lem:49}. In the third equality, we used Lemma~\ref{lem:30} thinking of \(\hate_k=\hats_k e_k\).
\end{proof}

\begin{lemma}
  \label{lem:21}
  We have \(\tilde e_k \sigma_{k+1} \tilde e_k \idem_\bolda = \tilde e_k \idem_\bolda \) for all \(k,k+1\in J\) and \(\bolda\in\Seq_{r,t}\) with \(a_{k+1}=a_{k+2}\).
\end{lemma}

\begin{proof}
  Let \(\idem_\bolda\) be such that \(a_k \neq a_{k+1}\), otherwise the claim is trivial. Then we compute
  \begin{equation}
    \label{eq:216}
    \begin{aligned}[t]
      \tilde e_k \sigma_{k+1} \tilde e_k \idem_\bolda&=
      \begin{multlined}[t]
        - Q_k e_k
        \sqrt{\frac{b_{k+1}}{b_k}}\sqrt{\frac{b_{k+1}}{b_{k+2}}}
        s_{k+1} \sqrt{\frac{b_{k+1}}{b_{k+2}}}
        \sqrt{\frac{b_{k+1}}{b_k}} \idempotente e_k
        Q_k \idem_\bolda \\-
        Q_k e_k
        \frac{b_{k+1}}{b_{k} b_{k+2}} \idempotente e_k
        Q_k \idem_\bolda
      \end{multlined}\\
      &= -Q_k  \sqrt\frac{1 }{b_{k+2}}  e_k \frac{b_{k+1}}{b_k} s_{k+1} b_{k+1} e_k  \sqrt\frac{1 }{ b_{k+2}}Q_k \idem_\bolda  -       \frac{1}{b_{k+2}} \tilde e_k \tilde e_k \idem_\bolda.
    \end{aligned}
  \end{equation}
  Now, we have
  \begin{equation}
    \label{eq:281}
    \begin{aligned}[t]
       & e_k \frac{b_{k+1}}{b_k} s_{k+1} b_{k+1} e_k \idempotente \idem_\bolda \\
       & \qquad = \idempotente e_k \frac{b_{k+1}b_{k+2}}{b_k} s_{k+1} e_k\idempotente\idem_\bolda -  e_k \frac{b_{k+1}}{b_k} e_k \idempotente\idem_\bolda\\
&\qquad = b_{k+2} e_k \frac{b_{k+1}}{b_k} s_{k+1} e_k \idempotente \idem_\bolda+ \delta  e_k \idempotente \idem_\bolda\\
      &\qquad= b_{k+2} e_k
      \left( \frac{-\beta_1^{\operatorname{rev}(a_k)} - y_k }{b_k} + \frac{\beta_1^{\operatorname{rev}(a_k)}
          + \beta_1}{b_k} \right) s_{k+1} e_k \idempotente \idem_\bolda+\delta  e_k \idempotente \idem_\bolda\\
      &\qquad = - b_{k+2} e_k s_{k+1}  e_k\idempotente \idem_\bolda + b_{k+2}(\beta_1^{\operatorname{rev}(a_k)}+\beta_1) e_k
      \frac{1}{b_k}s_{k+1} e_k \idempotente \idem_\bolda + \delta e_k\idempotente \idem_\bolda\\
      & \qquad = -b_{k+2} e_k \idempotente \idem_\bolda + \delta e_k \idempotente \idem_\bolda
    \end{aligned}
  \end{equation}
  by Lemma~\ref{lem:31}.
  Hence
  \begin{multline}
    \label{eq:283}
      \tilde e_k \sigma_{k+1} \tilde e_k \idem_\bolda  = 
      Q_k      \sqrt\frac{1}{ b_{k+2}} b_{k+2} e_k \sqrt\frac{1 }{
        b_{k+2}}Q_k\idem_\bolda \\  - \delta Q_k \sqrt\frac{1 }{ b_{k+2}}
      e_k \sqrt\frac{1 }{ b_{k+2}} Q_k\idem_\bolda
      - \frac{1}{b_{k+2}}      \tilde e_k \tilde e_k \idem_\bolda
       = \tilde e_k\idem_\bolda.
  \end{multline}
  The lemma is proved.
\end{proof}

\begin{lemma}
  \label{lem:25}
  We have \(\tilde e_k \sigma_{k-1} \tilde e_k\idem_\bolda = \tilde e_k \idem_\bolda\) for all  \(\bolda\) with \(a_{k-1}=a_k\).
\end{lemma}

\begin{proof}
  We suppose \(a_k\neq a_{k+1}\), since otherwise the claim is trivial.
  Then the term \(\tilde e_k \sigma_{k-1} \tilde e_k\idem_\bolda\) is equal to
  \begin{equation}
    \label{eq:149}
    - Q_k  e_k \sqrt{\frac{b_{k+1}}{b_{k-1}}} s_{k-1} \sqrt{\frac{b_{k+1}}{b_{k-1}}} \idempotente e_k Q_k  \idem_\bolda + Q_k  e_k \frac{b_{k+1} }{b_{k-1} b_{k}} \idempotente  e_k Q_k\idem_\bolda.
  \end{equation}
   Let us expand the first summand:
   \begin{equation}\label{eq:150}
     \begin{aligned}
&- Q_k  e_k \sqrt{\frac{b_{k+1}}{b_{k-1}}} s_{k-1} \sqrt{\frac{b_{k+1}}{b_{k-1}}} \idempotente e_k Q_k  \idem_\bolda \\
& \qquad=       -Q_k \sqrt\frac{1}{b_{k-1} }  e_k s_{k-1} {c_{k}} e_k  \sqrt\frac{1}{b_{k-1} }Q_k \idem_\bolda \\
      &\qquad = -Q_k \sqrt\frac{1}{b_{k-1} }  e_k (c_{k-1} s_{k-1} -1) e_k  \sqrt\frac{1}{b_{k-1}}\idem_\bolda\\
   & \qquad =   \tilde e_k \idem_\bolda - (\beta_1^\uparrow + \beta_1^\downarrow) Q_k \sqrt\frac{1}{b_{k-1} } e_k \sqrt\frac{1}{b_{k-1}} Q_k \idem_\bolda\\
   & \qquad \qquad+ (m+n) Q_k \sqrt\frac{1}{b_{k-1}} e_k \sqrt\frac{1}{b_{k-1}} Q_k \idem_\bolda,
 \end{aligned}
\end{equation}
where we used \(- c_{k-1} = b_{k-1} - (\beta_1^\uparrow + \beta_1^\downarrow)\).
  If we expand the second summand of \eqref{eq:149} we obtain:
  \begin{equation}
    \label{eq:152}
 Q_k  e_k \frac{b_{k+1} }{b_{k-1} b_{k}} \idempotente  e_k Q_k\idem_\bolda=   Q_k \sqrt\frac{1}{b_{k-1}} e_k \frac{c_{k}}{b_k} e_k \sqrt\frac{1}{b_{k-1}}Q_k \idem_\bolda .
  \end{equation}
  We use again \(c_k = - b_{k-1} + (\beta_1^\uparrow + \beta_1^\downarrow)\) and we get:
  \begin{equation}
    \label{eq:153}
    - (m+n) Q_k \sqrt\frac{1}{b_{k-1}} e_k \sqrt\frac{1}{b_{k-1}} Q_k \idem_\bolda+ (\beta_1^\uparrow+\beta_1^\downarrow) Q_k \sqrt\frac{1}{b_{k-1}}  e_k \sqrt\frac{1}{b_{k-1}} Q_k\idem_\bolda.
  \end{equation}
  Comparing \eqref{eq:150} and \eqref{eq:153} we get the claim.
\end{proof}

\begin{lemma}
  \label{lem:33}
  We have \( \thats_k \te_{k+1} \te_k  =  \ts_{k+1} \thate_k  \) for \(k,k+1\in J\).
\end{lemma}

\begin{proof}
We compute
  \begin{equation}
    \label{eq:187}
    \begin{aligned}[t]
      \thats_k \te_{k+1} \te_k & = - Q_k \hats_k \sqrt\frac{b_{k+1}}{b_k}\idempotente \sqrt\frac{b_{k+2}}{b_{k+1}} e_{k+1} \sqrt\frac{b_{k+2}}{b_{k+1}}\idempotente \sqrt\frac{b_{k+1}}{b_k} e_k Q_k \\
      & = - Q_k \hats_k \sqrt\frac{b_{k+1}}{b_k} \sqrt\frac{b_{k+2}}{b_{k+1}} e_{k+1} \sqrt\frac{b_{k+2}}{b_{k+1}} \sqrt\frac{b_{k+1}}{b_k} e_k Q_k  \\
      &= -Q_k \sqrt{b_{k+2}} \left(\frac{1}{b_{k+1}} \hats_k - \frac{1}{b_{k+1}}\hate_k \frac{1}{b_k}  \right) e_{k+1}e_k \sqrt{b_{k+2}} Q_k \\
      &= -\idempotente \sqrt\frac{b_{k+2}}{b_kb_{k+1} } \left(s_{k+1} \hate_{k} - \hate_k \frac{1}{c_{k+1}} e_{k+1}e_k \right) \sqrt\frac{b_{k+1} b_{k+2}}{b_k}\idempotente \\
      &= -\idempotente \sqrt\frac{b_{k+2}}{b_k b_{k+1}} \left(s_{k+1} \hate_{k} - \frac{1}{b_{k+2}} \hate_k \right) \sqrt\frac{b_{k+1} b_{k+2}}{b_k}\idempotente .
    \end{aligned}
  \end{equation}
  We should explain why we can omit the idempotents in the second equality.
Let  \(\theta_{k+1}\) be the idempotent projecting onto the generalized eigenspaces of \(y_{k+1}\idem_\bolda\) with eigenvalues different from \(-\beta^{a_k}_1\). Note that \(e_k \eta_k \idem_\bolda = e_k \theta_{k+1} \idem_\bolda\), and that \(\theta_{k+1}\) can be expressed as a polynomial in the variable \(y_{k+1} \idem_\bolda\). Now,
 we have 
 \begin{multline}
 \idempotente \hat s_k \idempotente e_{k+1} \idempotente e_k \idempotente\idem_\bolda= \idempotente \hat s_k  \eta_{k+1} e_{k+1} \eta_k \eta_{k+1}  e_k \idempotente\idem_\bolda = \idempotente \hat s_k  \theta_{k+2} e_{k+1} \theta_{k+1} \theta_{k+2}  e_k \idempotente \idem_\bolda
\\  = \idempotente \theta_{k+2} \hat s_k  \theta_{k+2} e_{k+1}   e_k \eta_{k+2} \theta_{k+2}\idempotente \idem_\bolda = \idempotente \hat s_k e_{k+1} e_k \idempotente \idem_\bolda.
\end{multline}
as claimed. On the other hand, we have
  \begin{equation}
    \label{eq:189}
    \begin{aligned}[t]
      \ts_{k+1}\thate_{k}  &= -Q_{k+1} s_{k+1} \sqrt\frac{b_{k+2}}{b_{k}}\hate_k Q_k + Q_{k+1} \frac{1}{b_{k+1}} \hate_k Q_{k+1} \\
      &=- Q_{k+1}\sqrt\frac{1}{b_k}s_{k+1} \hate_k \sqrt{b_{k+1}} Q_{k} + \idempotente   \sqrt\frac{1}{b_k b_{k+1}}\hate_k Q_k.
    \end{aligned}
  \end{equation}
  Here we can again omit the idempotent in the middle, since 
  \begin{equation}
  \idempotente s_{k+1} \idempotente \hate_k \idempotente  = \idempotente s_{k+1} \eta_k \eta_{k+2} \hate_k \idempotente   = \idempotente \eta_k s_{k+1} \hate_k \eta_{k+2} \idempotente  = \idempotente s_{k+1} \hate_k \idempotente.
  \end{equation}
  So \eqref{eq:187} and \eqref{eq:189} are the same and the lemma follows.
\end{proof}

\begin{lemma}
  \label{lem:34}
  We have \(\ts_k \thate_{k+1} \te_k = \thats_{k+1} \te_k \) for \(k,k+1\in J\).
\end{lemma}

\begin{proof}
  We compute
  \begin{equation}
    \label{eq:190}
    \begin{aligned}[t]
      \ts_k \thate_{k+1} \te_k & = \left(-Q_k s_k Q_k  + \frac{1}{b_k} \idempotente \right)Q_{k+2} \hate_{k+1} \sqrt\frac{b_{k+2}}{b_{k}}\idempotente  e_k Q_{k+1}\\
      &= Q_k\sqrt{ b_{k+2}}\left(- s_k \frac{1}{b_k} +
        \frac{1}{b_k b_{k+1}} \right) \hate_{k+1}
       e_k \sqrt{b_{k+2}} Q_k\\
      &= -Q_k\sqrt{b_{k+2}} \frac{1}{b_{k+1}} s_k \hate_{k+1}
       e_k \sqrt{b_{k+2}} Q_k\\
       & = - Q_{k+1}  \sqrt\frac{1}{b_k} \hats_{k+1} e_k \sqrt{b_{k+2}} Q_k .
    \end{aligned}
  \end{equation}
  In the second line, we could omit the idempotent as in the previous proof. On the other side, we have
  \begin{equation}
    \label{eq:191}
      \thats_{k+1} \te_k = -Q_{k+1} \hats_{k+1}
      \sqrt\frac{b_{k+2}}{b_k} e_k
      Q_k     = -Q_{k+1}\sqrt\frac{ 1}{b_k}
      \hats_{k+1} e_k \sqrt{b_{k+2}} Q_k.
  \end{equation}
  hence they coincide.
\end{proof}

\begin{lemma}
  \label{lem:37}
  We have \(\thats_{k+1} \te_{k} \te_{k+1} =  \ts_{k} \thate_{k+1}\).
\end{lemma}

\begin{proof}
  We compute
  \begin{equation}
    \label{eq:217}
    \begin{aligned}[t]
     \thats_{k+1} \te_{k} \te_{k+1}  &= - Q_{k+1}  \hats_{k+1} Q_{k+1} Q_k  e_{k} Q_k Q_{k+1} e_{k+1} Q_{k+1}\\
      &= - Q_{k+1}  \sqrt\frac{1}{b_k} \left(b_{k+1} \hats_{k+1} - \hate_{k+1} \right) e_{k} e_{k+1} \sqrt\frac{1}{b_k } Q_{k+1} \\
      &=-Q_k \sqrt{ b_{k+2}} \hats_{k+1} e_k e_{k+1} \sqrt\frac{1}{b_k} Q_{k+1} + Q_{k+1}   \sqrt\frac{1}{b_k}\hate_{k}\sqrt\frac{1}{b_k }Q_{k+1} 
    \end{aligned}
  \end{equation}
  and
  \begin{equation}
    \label{eq:218}
    \begin{aligned}[t]
      \ts_{k}\thate_{k+1}  & = - Q_{k} s_{k} \sqrt\frac{b_{k+2}}{b_{k}} \hate_{k+1} Q_{k+1} + Q_{k+1} \frac{1}{b_{k}} \hate_{k+1} Q_{k+1} \\
      &=-Q_k  \sqrt{b_{k+2}} s_{k} \hate_{k+1} Q_{k+1} +  Q_{k+1} \sqrt\frac{1}{b_k}\hate_k \sqrt\frac{1}{b_k} Q_{k+1} .\\
    \end{aligned}
  \end{equation}
  So \eqref{eq:217} and \eqref{eq:218} agree.
\end{proof}

\begin{lemma}
  \label{lem:38}
  We have \(\ts_{k+1} \thate_{k} \te_{k+1} = \thats_{k} \te_{k+1}\).
\end{lemma}

\begin{proof}
  We compute
  \begin{equation}
    \label{eq:219}
    \begin{aligned}[t]
      \ts_{k+1} \thate_{k} \te_{k+1} & = \left(- Q_{k+1} s_{k+1} Q_{k+1} + \frac{1}{b_{k+1}}\idempotente \right) Q_k \hate_k \sqrt\frac{b_{k+2}}{b_{k}} \idempotente e_{k+1} Q_{k+1}\\
      & = Q_{k+1} \sqrt\frac{1}{b_k} \left( -s_{k+1} b_{k+2} + 1\right) \hate_k e_{k+1} \sqrt\frac{1}{b_k} Q_{k+1}\\
      & = - Q_k \sqrt{b_{k+2}}  s_{k+1} \hate_k e_{k+1} \sqrt\frac{1}{b_k} Q_{k+1}\\
    \end{aligned}
  \end{equation}
  and
  \begin{equation}
    \label{eq:220}
      \thats_{k} \te_{k+1} = -Q_k\hats_{k} \sqrt\frac{b_{k+2}}{b_{k}} \idempotente e_{k+1} Q_{k+1} = - Q_k \sqrt{b_{k+2}} \hats_{k}  e_{k+1} \sqrt\frac{1}{b_k} Q_{k+1}.
  \end{equation}
  Since they agree the lemma is proved.
\end{proof}

\begin{lemma}
  \label{lem:24}
  The braid relation \(\sigma_k \sigma_{k+1} \sigma_k = \sigma_{k+1} \sigma_k \sigma_{k+1}\) holds.
\end{lemma}

\begin{proof}
  Let \(\bolda\) be a sequence with \(a_k=a_{k+1}=a_{k+2}\). All the summands of the formulas in this proof are supposed to be multiplied on the right (or on the left) with \(\idem_\bolda\), but we omit it for the sake of clearness. We compute 
  \begin{subequations}
    \begin{align}
        \sigma_k \sigma_{k+1} \sigma_k = & -Q_ks_k\sqrt\frac{b_{k+2}}{b_k}s_{k+1}\sqrt\frac{b_{k+2}}{b_k}s_k
        Q_k \label{eq:125}\\
        &  + Q_k s_k \sqrt\frac{b_{k+2}}{b_k} s_{k+1}
        \frac{Q_{k+1}}{b_k}\label{eq:126}\\
        &  + Q_k s_k \frac{1}{b_k}s_kQ_k\label{eq:127} \\
        &  - Q_ks_k\frac{Q_k}{b_k b_{k+1}} \label{eq:128}\\
        &  +\frac{Q_{k+1}}{b_k}s_{k+1}\sqrt\frac{b_{k+2}}{b_k} s_kQ_k\label{eq:129}\\
        &  -\frac{Q_{k+1}}{b_k} s_{k+1}\frac{Q_{k+1}}{b_k}\label{eq:130}\\
        &  -\frac{Q_k}{b_k b_{k+1}} s_k Q_k \label{eq:131}\\
        &  + \frac{1}{b_k^2 b_{k+1}} \idempotente.\label{eq:132}
    \end{align}
  \end{subequations}
  Recall that the idempotent \(\idempotente\) commutes with \(s_k\) and \(s_{k+1}\), so it is sufficient if it appears once in every summand (in the term \(Q_k\) or \(Q_{k+1}\)). Now we consider the various pieces.
  \begin{eqnarray}
    \label{eq:133}
      &\text{\eqref{eq:125}} \; =\; -Q_k \sqrt{b_{k+2}} s_k \frac{1}{b_k}
      s_{k+1}s_{k} \sqrt{b_{k+2}}Q_k&\\
      & \;=\; - Q_k \sqrt{b_{k+2}} \frac{1}{b_{k+1}} s_k s_{k+1}s_k \sqrt{b_{k+2}}Q_k - Q_k \frac{ \sqrt{b_{k+2}}}{b_kb_{k+1}} s_{k+1}s_k \sqrt{b_{k+2}} Q_k.
    \nonumber
  \end{eqnarray}
  \begin{equation}
    \label{eq:120}
    \begin{aligned}
     \eqref{eq:127}  + \eqref{eq:131}  = Q_k \frac{1}{b_{k+1}} Q_k + \frac{Q_k}{b_k b_{k+1}} s_k Q_k  -  \frac{Q_k}{b_k b_{k+1}} s_k Q_k  = \frac{1}{b_k} \idempotente.
    \end{aligned}
  \end{equation}
  Moreover, \eqref{eq:126} + \eqref{eq:130} equals the following
  \begin{multline}
    \label{eq:123}
      \frac{Q_k}{b_{k+1}} s_k \sqrt\frac{b_{k+2}}{b_k} s_{k+1} Q_{k+1} + \frac{Q_k}{b_{k}b_{k+1}}\sqrt\frac{b_{k+2}}{b_k} s_{k+1} Q_{k+1} - \frac{Q_{k+1}}{b_k} s_{k+1} \frac{Q_{k+1}}{b_k} \\
       = \frac{Q_{k+1}}{\sqrt{b_k}} s_k s_{k+1} \frac{Q_{k+1}}{\sqrt{b_k}}.
  \end{multline}
  Moreover,
  \begin{equation}
    \label{eq:124}
    \begin{aligned}[t]
      \eqref{eq:129} & = Q_{k+1} s_{k+1} \sqrt\frac{b_{k+2}}{b_k} s_k
      \frac{Q_k}{b_{k+1}} + Q_{k+1} s_{k+1} \sqrt\frac{b_{k+2}}{b_k}
      \frac{Q_k}{b_k b_{k+1}} \\
      & = \frac{Q_{k+1}}{\sqrt{b_k}} s_{k+1} s_k \frac{Q_{k+1}}{\sqrt{b_k}}
       + \frac{Q_{k+1}}{b_k} s_{k+1} \frac{Q_{k+1}}{b_k}.
    \end{aligned}
  \end{equation}
  and
  \begin{equation}
    \label{eq:145}
    \eqref{eq:128} = - \frac{Q_k}{b_{k+1}} s_k \frac{Q_k}{b_{k+1}} - \frac{1}{b_k^2 b_{k+1}} \idempotente.
  \end{equation}
  So we can rewrite:
  \begin{equation}
    \label{eq:146}
    \sigma_k \sigma_{k+1} \sigma_k = \eqref{eq:133} + \eqref{eq:123} + \frac{1}{b_k}\idempotente - \frac{Q_k}{b_{k+1}} s_k \frac{Q_k}{b_{k+1}} + \eqref{eq:124}.
  \end{equation}

On the other side, the r.h.s. is:
  \begin{subequations}
    \begin{align}
      \sigma_{k+1} \sigma_k \sigma_{k+1} = & - Q_{k+1}s_{k+1} \sqrt\frac{b_{k+2}}{b_k} s_k \sqrt\frac{b_{k+2}}{b_k} s_{k+1} Q_{k+1} \label{eq:135} \\
      & + Q_{k+1} s_{k+1} \sqrt\frac{b_{k+2}}{b_k} s_k \frac{Q_k}{b_{k+1}} \label{eq:136}\\
      & + Q_{k+1} s_{k+1} \frac{b_{k+2}}{b_k b_{k+1}} s_{k+1} Q_{k+1} \label{eq:137}\\
      & - Q_{k+1} s_{k+1} \frac{Q_{k+1}}{b_k b_{k+1}} \label{eq:138}\\
      & + \frac{Q_k}{b_{k+1}}s_k \sqrt\frac{b_{k+2}}{b_k} s_{k+1}Q_{k+1} \label{eq:139}\\
      & - \frac{Q_k}{b_{k+1}} s_k \frac{Q_k}{b_{k+1}} \label{eq:140}\\
      & - \frac{Q_{k+1}}{b_k b_{k+1}} s_{k+1} Q_{k+1} \label{eq:141}\\
      & + \frac{1}{b_k b_{k+1}^2} \idempotente \label{eq:142}.
    \end{align}
  \end{subequations}
  Again, we consider some pieces: First of all we have
  \begin{equation}
    \label{eq:134}
    \begin{aligned}
      \eqref{eq:135} & = -\frac{Q_{k+1}}{\sqrt{b_k}} s_{k+1}  s_k s_{k+1} \frac{b_{k+1} Q_{k+1}}{\sqrt{b_k}} - \frac{Q_{k+1}}{\sqrt{b_k}} s_{k+1}  s_k  \frac{Q_{k+1}}{\sqrt{b_k}}.
    \end{aligned}
  \end{equation}
  Moreover
  \begin{equation*}
    \begin{aligned}
      \eqref{eq:137} & = \frac{Q_{k+1}}{b_k} b_{k+1} s_{k+1}\frac{1}{b_{k+1}} s_{k+1} Q_{k+1} + \frac{Q_{k+1}}{b_k b_{k+1}} s_{k+1} Q_{k+1}\\
      & = \frac{Q_{k+1} b_{k+1}}{b_k b_{k+2}} Q_{k+1} + \frac{Q_{k+1} }{b_k b_{k+2}} s_{k+1} Q_{k+1}+ \frac{Q_{k+1}}{b_k b_{k+1}} s_{k+1} Q_{k+1}\\
      & = \frac{Q_{k+1} b_{k+1}}{b_k b_{k+2}} Q_{k+1} + \frac{Q_{k+1} }{b_k } s_{k+1}\frac{ Q_{k+1}}{b_{k+1}} - \frac{Q_{k+1}^2}{b_k b_{k+1} b_{k+2}}+ \frac{Q_{k+1}}{b_k b_{k+1}} s_{k+1} Q_{k+1}\\
    \end{aligned}
  \end{equation*}
  and hence
  \begin{equation}
    \label{eq:144}
      \eqref{eq:137} + \eqref{eq:138} + \eqref{eq:141} + \eqref{eq:142} = \frac{1}{b_k} \idempotente.
  \end{equation}
  So we obtain
  \begin{multline*}
      \sigma_k \sigma_{k+1} \sigma_k - \sigma_{k+1} \sigma_k
      \sigma_{k+1} = \eqref{eq:133} - \eqref{eq:134} +
      \eqref{eq:123} + \eqref{eq:124} -
      \eqref{eq:136} - \eqref{eq:139}\\
       = - \frac{1}{b_k} \frac{Q_{k+1}}{\sqrt{b_k}} s_{k+1} s_k Q_k \sqrt{b_{k+2}} + \frac{Q_{k+1}}{\sqrt{b_k}} s_{k+1} s_k \frac{Q_{k+1}}{\sqrt{b_k}} +  \frac{Q_{k+1}}{b_k} s_{k+1} \frac{Q_{k+1}}{b_k} = 0. \qedhere
  \end{multline*}
\end{proof}

\begin{lemma}
  \label{lem:35}
  The braid relation \(\thats_k \thats_{k+1} \ts_k \idem_\bolda= \ts_{k+1} \thats_k \thats_{k+1} \idem_\bolda\) holds for all \(\bolda\) with \(a_k=a_{k+1}\neq a_{k+2}\).
\end{lemma}

\begin{proof}
We omit the idempotent \(\idem_\bolda\) in the following formulas (every summand should be multiplied by \(\idem_\bolda\) from the right). First we observe that \(\idempotente \hat s_k \idempotente \hat s_{k+1} \idempotente = \idempotente \hat s_k \hat s_{k+1} \idempotente\). Indeed, since \(y_{k+2}\) commutes with \(s_k\) and \(y_{k}\) commutes with \(s_{k+1}\), both the generalized eigenvalues of \(y_k\) and \(y_{k+2}\) in the middle have to be small. By Lemma~\ref{lem:2}, the generalized eigenvalue of \(y_{k+1}\) in the middle has also to be small, hence we can omit the idempotent \(\idempotente\) in the middle.

Now, we have
  \begin{equation}
    \label{eq:192}
    \begin{aligned}[t]
      \thats_k \thats_{k+1} \ts_k & = - Q_k \hats_k \sqrt\frac{b_{k+2}}{b_{k}} \hats_{k+1} \sqrt\frac{b_{k+2}}{b_{k+1}} \left(\sqrt\frac{b_{k+1}}{b_k} s_k Q_k - \frac{1}{b_k}\idempotente \right)\\
     & = - Q_k \sqrt{b_{k+2}} \hats_k  \hats_{k+1} \left( \frac{1}{b_k} s_k - \frac{1}{b_k b_{k+1}}\right) Q_k \sqrt{b_{k+2}}\\
     & = - Q_k \sqrt{b_{k+2}} \hats_k  \hats_{k+1} s_k \frac{1}{b_{k+1}} Q_k \sqrt{b_{k+2}}\\
    \end{aligned}
  \end{equation}
  and
  \begin{equation}
    \label{eq:193}
    \begin{aligned}[t]
      \ts_{k+1} \thats_k \thats_{k+1} &= \left(- Q_{k+1} s_{k+1} \sqrt\frac{b_{k+2}}{b_{k+1}} + \frac{1}{b_{k+1}} \right) \sqrt\frac{b_{k+1}}{b_k} \hats_k \sqrt\frac{b_{k+2}}{b_{k}} \hats_{k+1} Q_{k+1}\\
     & = Q_k \sqrt{b_{k+2}}\left(- \frac{1}{b_{k+1}} s_{k+1} b_{k+2}  + \frac{1}{b_{k+1}} \right) \hats_k \hats_{k+1} \frac{Q_{k+1}}{\sqrt{b_{k}}}\\
     & = -Q_k \sqrt{b_{k+2}} s_{k+1}  \hats_k \hats_{k+1} \frac{Q_{k+1}}{\sqrt{b_{k}}}.\\
    \end{aligned}
  \end{equation}
  The two terms \eqref{eq:192} and \eqref{eq:193} are the same because \(\hats_k \hats_{k+1} s_k=s_{k+1} \hats_k \hats_{k+1}\).
\end{proof}


\providecommand{\bysame}{\leavevmode\hbox to3em{\hrulefill}\thinspace}
\providecommand{\MR}{\relax\ifhmode\unskip\space\fi MR }
\providecommand{\MRhref}[2]{%
  \href{http://www.ams.org/mathscinet-getitem?mr=#1}{#2}
}
\providecommand{\href}[2]{#2}



\begin{thebibliography}{BLPW12}


\bibitem[AMR06]{MR2235339}
S.~Ariki, A.~Mathas and H.~Rui, \emph{Cyclotomic {N}azarov-{W}enzl algebras},
  Nagoya Math. J. \textbf{182} (2006), 47--134.


\bibitem[Betal94]{Beetal}
G. Benkart, M. Chakrabarti, T. Halverson, R. Leduc, C. Lee and J. Stroomer,
\emph{Tensor product representations of general linear groups and their connections with Brauer algebras},
{J. Algebra} {\bf 166} (1994), 529--567.

\bibitem[BLPW12]{BLPW}
T.~ Braden, A.~Licata, N.~Proudfoot and B.~Webster, \emph{Hypertoric category $\scr{O}$},
J Adv. Math., {\bf 231} (2012), 1487--1545.


\bibitem[B08a]{Bcenter}
J.~Brundan, \emph{Centers of degenerate cyclotomic Hecke algebras and parabolic category
   $\scr O$}, J Represent. Theory, {\bf 12}, (2008), 236--259.

\bibitem[B08b]{BSpringer}
\bysame \emph{Symmetric functions, parabolic category $\scr O$, and the Springer fiber}, Duke Math. J., {\bf 143}, (2008), 41--79.

\bibitem[BCNR14]{Betal}
J.~ Brundan, J.~ Comes, D.~ Nash and  A.~Reynolds, \emph{A basis theorem for the affine oriented Brauer category and its cyclotomic quotients} arXiv:1404.6574.  

\bibitem[BS11]{MR2781018}
J.~Brundan and C.~Stroppel, \emph{Highest weight categories arising from
  {K}hovanov's diagram algebra {III}: category {$\scr O$}}, Represent. Theory
  \textbf{15} (2011), 170--243.

\bibitem[BS12]{MR2955190}
\bysame, \emph{Gradings on walled {B}rauer algebras and {K}hovanov's arc
  algebra}, Adv. Math. \textbf{231} (2012), no.~2, 709--773.

\bibitem[CST10]{CST}
T.~Ceccherini-Silberstein, F.~Scarabotti and F.~Tolli, \emph{Representation theory of the symmetric groups}, {Cambridge Studies in Advanced Mathematics}, \textbf{121}, Cambridge University Press, (2010).

\bibitem[DRV14]{MR3192543}
Z.~Daugherty, A.~Ram, and R.~Virk, \emph{Affine and degenerate affine {BMW}
  algebras: the center}, Osaka J. Math. \textbf{51} (2014), no.~1, 257--285.
  
 \bibitem[DDS]{DDS}
 R.  Dipper, S. Doty and F. Stoll, {\emph The quantized walled Brauer algebra and mixed tensor space}. 
Algebr. Represent. Theory \text{bf 17} (2014), no. 2, 675–701.  
  
\bibitem[ES14a]{ES}
M.~Ehrig and C.~Stroppel, \emph{Nazarov-Wenzl algebras, coideal subalgebras and categorified skew Howe duality}, arXiv:1310.1972. 

\bibitem[ES14b]{ESosp}
\bysame \emph{Schur-Weyl duality for the Brauer algebra and the ortho-symplectic Lie superalgebra}, arXiv:1412.7853.

\bibitem[ES15]{ESBrauer}
\bysame \emph{Koszul gradings on Brauer algebras.},  arXiv:1504.03924. 

\bibitem[F97]{MR1464693}
W.~Fulton, \emph{Young tableaux}, London Mathematical Society Student Texts,
  vol.~35, Cambridge University Press, Cambridge, 1997.

\bibitem[H96]{Ha}
T. Halverson, \emph{Characters of the centralizer algebras of
mixed tensor representations of
$GL(r,\mathbb{C})$ and the Quantum Group $U_q(\mathfrak{gl}(r,\mathbb{C}))$},
Pacific J. Math. {\bf 174} (1996), 359--410.

\bibitem[H08]{MR2428237}
J.~E. Humphreys, \emph{Representations of semisimple {L}ie algebras in the
  {BGG} category {$\scr{O}$}}, Graduate Studies in Mathematics, {\bf 94},
  AMS, Providence, (2008).

\bibitem[Ko89]{K}
K. Koike, \emph{On the decomposition of tensor products of the representations of classical groups:
by means of universal characters},  Advances Math., {\bf 74} (1989), 57--86.

\bibitem[Ko93]{Kosuda}
M. Kosuda and J. Murakami, \emph{Centralizer algebras of the mixed tensor representations of quantum group $U_q(\mathfrak{gl}(n,\mathbb{C})$}. Osaka J. Math. {\bf 30} (1993), no. 3, 475–507.

\bibitem[L14]{Li}
G.~Li, \emph{ A KLR Grading of the Brauer Algebras}, arXiv:1409.1195.

\bibitem[RS14]{RuiSu}
H. Rui and Y. Su,  \emph{Affine walled Brauer algebras}, arXiv:1305.0450.

\bibitem[Naz96]{MR1398116}
M.~Nazarov, \emph{Young's orthogonal form for {B}rauer's centralizer algebra},
  J. Algebra \textbf{182} (1996), no.~3, 664--693.

\bibitem[N07]{N}
P. Nikitin,
\emph{The centralizer algebra of the diagonal action of the group $GL_n(\mathbb{C})$
in a mixed tensor space},
{J. Math. Sci.} {\bf 141} (2007), 1479--1493.


\bibitem[Sa14]{MR3244645}
A.~Sartori, \emph{The degenerate affine walled {B}rauer algebra}, J. Algebra
  \textbf{417} (2014), 198--233.

\bibitem[St09]{StSpringer}
C.~Stroppel, \emph{Parabolic category $\mathcal{O}$, perverse sheaves on Grassmannians, Springer
   fibres and Khovanov homology}, Compos. Math., {\bf 145} (2009), 954--992.

\bibitem[T89]{T}
 V. Turaev,
\emph{Operator invariants of matrices and R-matrices},
{Izv. Akad. Nauk SSSR} {\bf 53} (1989)
1073--1107.
\end{thebibliography}

\end{document}